\documentclass[reqno,12pt]{amsart}
\usepackage{amsmath, amssymb, amsthm, amsfonts} 
\usepackage[english]{babel}
\usepackage{bbm}
\usepackage{graphicx}
\usepackage{subcaption} 
\usepackage{url}
\usepackage{soul}
\usepackage{epstopdf}
\usepackage[ruled,vlined]{algorithm2e}
\usepackage[a4paper,bindingoffset=0.5cm,left=2cm,right=2cm,top=2.5cm,bottom=2cm,footskip=.8cm]{geometry}
\usepackage{rotating}
\usepackage{amsbsy,enumerate,enumitem}
\usepackage{comment}
\usepackage{mathrsfs}
\usepackage{hyperref}
\usepackage{geometry}

\newcommand{\e}{\mathbb{E}}

\newcommand{\bs}{\boldsymbol}

\newcommand{\vb}{\vspace{3.2mm}}
\renewcommand{\hat}{\widehat}

\newcommand{\defD}{{\mathfrak D}}
\newcommand{\defDmin}{{\mathfrak D}_-}
\newcommand{\defDplus}{{\mathfrak D}_+}
\newcommand{\WS}[1]{{\color{blue}[WS: #1]}}

\newcommand{\vertiii}[1]{{\left\vert\kern-0.25ex\left\vert\kern-0.25ex\left\vert #1 \right\vert\kern-0.25ex\right\vert\kern-0.25ex\right\vert}}

\allowdisplaybreaks

\newcommand{\ind}[1]{{\bs 1}_{\{#1\}}}

\allowdisplaybreaks

\newtheorem{lemma}{Lemma}

\newtheorem{theorem}{Theorem}
\newtheorem{remark}{Remark}

\newtheorem{proposition}{Proposition}

\usepackage[utf8]{inputenc}
\usepackage{type1cm}         
\usepackage{multicol}        
\usepackage[bottom]{footmisc}

\usepackage{newtxtext}       %
\usepackage{newtxmath}       

\usepackage{pgfplots}

\usepackage{tikz}
\usepackage[mode=buildnew]{standalone}
\usetikzlibrary{automata,positioning}
\usetikzlibrary{decorations.pathreplacing,decorations.markings,shapes.geometric}
\usetikzlibrary{shapes,arrows}
\usetikzlibrary{backgrounds,calc,positioning}

\usepackage{pgfplots}
\pgfplotsset{compat=1.9}

\begin{document}

	\title[Finite-capacity MAP-driven queues]{Exact analysis of transient behavior\\ of finite-capacity MAP-driven queues}
\author[M. Mandjes, D. Rutgers, {\tiny and} W. Scheinhardt]{Michel Mandjes, Dani\"el Rutgers {\tiny and} Werner Scheinhardt}
	
	\begin{abstract} This paper studies the workload distribution of a finite-capacity queue driven by a spectrally one-sided Markov additive process (MAP). Our main result provides the Laplace–Stieltjes transform of the workload at an exponentially distributed time, thereby uniquely characterizing its transient distribution. The proposed approach combines several decompositions with established fluctuation-theoretic results for spectrally one-sided L\'evy processes. For the special case of Markov-modulated compound Poisson input, we additionally derive results for the idle time and the cumulative amount of lost work. We conclude this paper with a series of numerical experiments.

\vb

\noindent
{\sc Keywords.} Markov additive processes $\circ$ L\'evy processes $\circ$ fluctuation theory $\circ$ queueing $\circ$ finite capacity $\circ$ Laplace transforms $\circ$ overshoot distribution $\circ$ Markov modulation

\vb

\noindent
{\sc Affiliations.} 
MM is with Mathematical Institute, Leiden University, P.O. Box 9512,
2300 RA Leiden,
The Netherlands. He is also affiliated with Korteweg-de Vries Institute for Mathematics, University of Amsterdam, Amsterdam, The Netherlands; E{\sc urandom}, Eindhoven University of Technology, Eindhoven, The Netherlands; Amsterdam Business School, Faculty of Economics and Business, University of Amsterdam, Amsterdam, The Netherlands. 

\noindent
DR is with Mathematical Institute, Leiden University, P.O. Box 9512,
2300 RA Leiden,
The Netherlands. 

\noindent
WS is with Faculty of Electrical Engineering, Mathematics and Computer Science, University of Twente, 
P.O. Box 217, 7500 AE  Enschede, The Netherlands.

\noindent Date: {\it \today}.

\vb

\noindent
{\sc Acknowledgments.} 
MM's research has been partly funded by the NWO Gravitation project N{\sc etworks}, grant number 024.002.003.

\vb

\noindent
{\sc Email.} {\scriptsize \url{m.r.h.mandjes@math.leidenuniv.nl}, \url{d.t.rutgers@math.leidenuniv.nl}, \url{w.r.w.scheinhardt@utwente.nl}}.

\vb

\noindent
{\sc Corresponding author.} Dani\"el Rutgers.

	\end{abstract}

	\maketitle

    \newpage

\section{Introduction}\label{sec: intro}
In this paper, we study the workload of a queue that is driven by a {\it spectrally-positive Markov additive process} (MAP) $Y\equiv (Y(t))_{t\geqslant 0}$, with the special feature that an upper boundary $K>0$ is imposed on the workload level. Let $V\equiv (V(t))_{t\geqslant 0}$ denote the resulting doubly-reflected workload process, so that $0 \leqslant V(t)\leqslant K$ for all $t\geqslant 0$. 
Our objective is to uniquely characterize, for any initial workload $x \in [0,K]$, the distribution of $V(t)$ at an arbitrary time $t$. We achieve this characterization by identifying the distribution of $V(T_\beta)$, where $T_\beta$ is an exponentially distributed random variable with parameter $\beta>0$, sampled independently of everything else.
The driving process $Y$, initialized at $Y(0)=0$, is commonly referred to as the {\it free process}, or alternatively as the queue's {\it input process}.

It is noted that within our finite-capacity framework, the above formulation also covers the case where $Y$ is a spectrally-{\it negative} MAP. In this setting, the reflected process $(K-V(t))_{t\geqslant 0}$ is a doubly-reflected spectrally-positive MAP, and hence the results of this paper apply directly. This means that in this paper's main results we characterize the workload distribution of a finite-capacity queue driven by a {\it spectrally one-sided} MAP.

\medskip

{\it Model.} The spectrally-positive MAP is defined as follows. Let $J\equiv (J(t))_{t\geqslant0}$ be a continuous-time Markov chain on  $\defD :=\{1,\ldots,d\}$, which serves as the \emph{background process}, or {\it modulating process}. We do not impose any structure on this Markov chain; in particular we do not require it to be irreducible. Its transition matrix is given by
\[Q \equiv (q_{ij})_{i,j=1}^d,\]
with the convention $q_i := -q_{ii} = \sum_{j\neq i} q_{ij}$ for any $i\in\defD $. If the background process resides in state $i\in\defD $, then the workload process is driven by a spectrally-positive L\'evy process $Y_i\equiv (Y_i(t))_{t\geqslant0}$, i.e., a L\'evy process without negative jumps, that is characterized via its Laplace exponent $\varphi_i(\cdot)$. Examples of spectrally-positive L\'evy processes are a Brownian motion with drift, a compound Poisson process or a combination of the two (but the class is considerably broader). In our model the workload can also increase at transitions of the background process: in case of a transition from $i$ to $j\not=i$, a job arrives with a size that is distributed as the non-negative random variable $B_{ij}$ with Laplace-Stieltjes transform (LST) ${\mathscr B}_{ij}(\cdot)$. The evident independence properties are imposed. 

The process $V$, with initial condition $V(0)=x\in[0,K]$, is then the doubly-reflected version of $(x+Y(t))_{t\geqslant 0}$, where we refer to $K$ as the queue's {\it capacity}. As discussed in detail in e.g.\ \cite[\S 1]{AAGP}, this means that one can write, for any $t\geqslant 0$, 
\begin{equation}\label{eq: regu} 
    V(t) = x + Y(t) + U^-(t) - U^+(t),
\end{equation}
where $U^-\equiv (U^-(t))_{t\geqslant 0}$ and $U^+\equiv (U^+(t))_{t\geqslant 0}$ are non-decreasing right-continuous processes, sometimes referred to as {\it regulators}, that solve the associated (double sided) {\it Skorokhod problem}. Specifically,
the process $U^-$ can only increase when $V$ is at the lower boundary 0, and $U^+$ only
when $V$ is at the upper boundary $K$.
Note that we work with \emph{partial rejection}: if an arriving job does not fully fit into the buffer of capacity $K$, the workload level is truncated at $K$. It is crucial in our setup that the state space $\defD$ of the background process is \emph{finite}. MAPs with a countable state space have been studied as well; see, for example, associated laws of large numbers and central limit theorems in \cite{KR}.

\medskip

{\it Objective.} This paper aims to evaluate, for all $i,j\in\defD $ and $x\in[0, K]$, the LST
\begin{equation}\label{eq: def chi_ij}
    \chi_{ij}(x) \equiv \chi_{i,j}(x,\alpha,\beta) := \e_{x,i}\left( e^{-\alpha V(T_{\beta})} \bs{1}_{\{J(T_{\beta})=j\}} \right),
\end{equation}
for $\alpha\geqslant 0$ and $\beta > 0;$ the subscripts in the rightmost expression indicate that the initial condition is $V(0)=x$ and $J(0)=i.$ This object uniquely describes the distribution of the queue's transient workload. More particularly, it can, up to a multiplicative constant, be seen as a double transform; indeed, we can write
\begin{align*}
    \frac{\chi_{ij}(x)}{\beta} &= \int_0^\infty e^{-{\beta} t} \,\e_{x,i}\left( e^{-\alpha V(t)} \bs{1}_{\{J(t)=j\}} \right){\rm d}t\\&= \int_{0}^\infty \int_{y\in[0,K]} e^{-\alpha y-\beta t} \,{\mathbb P}_{x,i}\left( V(t)\in {\rm d}y, J(t)=j \right)\,{\rm d}t .
\end{align*}
This means that the probability measure ${\mathbb P}_{x,i}\left( V(t)\in {\rm d}y, J(t)=j \right)$ can be found  \cite{AW,dI} from the transform $\chi_{ij}(x,\alpha,\beta)$ by numerical Laplace inversion (with respect to $\alpha$ and $\beta$, that is); it is noted that \cite{dI} provides a technique particularly suited to handle the {\it double} inversion. 

\medskip

{\it Literature.} We proceed by giving a non-exhaustive overview of the related literature. 
Queues without a cap on the workload have been much more intensively studied than their finite-capacity counterparts, but the latter topic has attracted substantial attention, too. Within the branch of the literature on finite-capacity queues, early work focused on the characterization of the stationary workload distribution of M/G/1-type systems~\cite{JWC,TAK}, with subsequent studies addressing the time-dependent case~\cite{R1,R2}.  In later work, the finite-capacity queue was cast into the more general framework of a stochastic process under double-sided Skorokhod reflection \cite{AND,HOU,KRU}; in this framework the finite-capacity M/G/1 queue is interpreted as a compound Poisson process with upward jumps and negative drift upon which double reflection has been imposed. An in-depth survey on the broader class of doubly-reflected L\'evy processes can be found in \cite{AAGP}. 

Most of the literature on the workload in the Markov-modulated compound Poisson-fed queue pertains to the case of infinite-capacity and stationarity. There is a one-to-one correspondence to an associated Cram\'er-Lundberg model under regime switching, which renders the results of e.g.\ \cite{IV,DM,MKD} directly applicable; see also the general frameworks \cite{DIEK,MKD} and the textbook treatments in e.g.\ \cite{ASM2,AA}. The stationary distribution in the finite-capacity case is briefly discussed in \cite[\S 16]{AAGP}. Regarding the workload at an exponentially distributed time, we explicitly mention the contributions \cite{IV2}, which studies Markov-modulated Brownian motion under double reflection, and \cite[\S 3.3]{IV}, which treats MAPs under double reflection in the special case that the workload process starts at level $0$ or $K$ and that the background process is irreducible.  

\medskip

{\it Contributions.} 
Our work is novel in that it combines the elements of (i)~spectrally one-sided Markov additive input, without imposing any requirements on the chain structure of the underlying background process, (ii)~finite capacity (i.e., a cap on the workload), (iii)~time-dependent behavior, and (iv) a general starting level. In more detail, the contributions of this paper are the following:
\begin{itemize}
    \item[$\circ$] We succeed in providing a procedure that produces,  based on a series of carefully constructed decompositions, for any $i,j\in\defD $ and $x\in[0, K]$, the LST
    $\chi_{ij}(x)$ as defined in \eqref{eq: def chi_ij}, in terms of the $2d^2$ constants $\chi_{ij}(0)$ and $\chi_{ij}(K)$.  We then provide a system of  linear equations that uniquely characterizes these constants. 
    Notably, in our analysis we allow for background states to be `subordinator states', i.e., states $i$ for which the underlying L\'evy process $Y_i$ is almost surely non-decreasing, for which we rely on ideas from \cite{MKD}. 
    \item[$\circ$] In the case that $Y$ is a Markov-modulated compound Poisson process, we point out how this procedure can be generalized to also cover the total amount of work lost up to $T_{\beta}$, as well as the total amount of idle time up to $T_{\beta}$.
    \item[$\circ$] As a by-product, our approach can be used to compute the moments of the workload process and the probability that the system is empty at a given point in time. Both quantities can be obtained directly, without the need to derive them from the LST.
\end{itemize}

{\it Organization.} The paper begins in Section~\ref{sec: prelim} with a collection of results for the case $d=1$ (that is, when $Y$ is a spectrally-positive L\'evy process), which are used extensively throughout the remainder of the paper. In Section~\ref{sec: analysis}, we express $\chi_{ij}(x)$ in terms of a set of auxiliary objects, whose explicit evaluation is carried out in Section~\ref{sec: aux}. Section~\ref{sec: extensions} then analyzes the cumulative amount of lost work and the total idle time up to $T_\beta$ in the case where $Y$ is of Markov-modulated compound Poisson type. Finally, numerical experiments are presented in Section~\ref{sec: num}.

\section{Preliminaries}\label{sec: prelim}
In this section, we consider the case $d=1$, that is, a workload process being fed by a single spectrally-positive L\'evy process $Y$; cf.\ the analysis in \cite{PI}. We treat this L\'evy-driven case first because several of the results obtained here will be used in later sections when analyzing the more general MAP-driven workload model.

\medskip

In our approach a crucial role is played by the following stopping times: for $u_-, u_+ \geqslant 0$,
\begin{align*}
    \sigma(u_-) &:= \inf\{t \geqslant 0: -Y(t) > u_-\} \\
    \tau(u_+) &:= \inf\{t \geqslant 0: Y(t) > u_+\}.
\end{align*}
These stopping times should be interpreted as the first time the free process $Y$ drops below the level $u_-$ or \textit{strictly} exceeds the level $u_+$, respectively.
We also define what we from here on call {\it hitting probabilities}, which include the expiration of the `exponential clock' $T_\beta$ (a concept often referred to as `killing'). We define, for $u_-, u_+ \geqslant 0$ with $u_- + u_+ > 0$ and $\beta > 0$,
\begin{align*}
    \delta_{-}(u_-, u_+, \beta) &:= \mathbb{P}\left(\sigma(u_-) \leqslant \min\{\tau(u_+), T_{\beta}\}\right) \\
    \delta_{+}(u_-, u_+, \beta) &:= \mathbb{P}\left(\tau(u_+) \leqslant \min\{\sigma(u_-), T_{\beta}\}\right).
\end{align*}
The hitting probability $\delta_{-}(u_-, u_+, \beta)$ thus denotes the probability that the process $Y$ drops below the negative level $-u_-$ before it exceeds the positive level $u_+$ and before being killed; the hitting probability $\delta_{+}(u_-, u_+, \beta)$ has an analogous interpretation.

Let $\varphi(\alpha) := \log \mathbb{E}e^{-\alpha Y(1)}$ denote the Laplace exponent of $Y$, and let $\psi(\cdot)$ denote its right-inverse. The {\it (primary) scale function} $W^{(\beta)}(\cdot):\mathbb{R}\to[0,\infty)$ is then defined as the function whose Laplace transform satisfies
\begin{equation*}
    \int_0^\infty e^{-\alpha y} W^{(\beta)}(y)\,\mathrm{d}y = \frac{1}{\varphi(\alpha) - \beta},
\end{equation*}
for $\alpha \geqslant 0$ and $\beta>0$ whenever $\varphi(\alpha) > \beta$; we define $W^{(\beta)}(y):=0$ for $y<0$. We also work with the \textit{secondary scale function}
\begin{equation*}
    Z^{(\beta)}(u) := 1 + \beta\int_0^u W^{(\beta)}(y)\,\mathrm{d}y.
\end{equation*}
For a full discussion on the primary and secondary scale functions of spectrally one-sided L\'evy processes, which play a key role in fluctuation theory, we refer to e.g.\ \cite[\S VIII.3]{KYP}, \cite[\S VII.2]{BER}, and \cite{KUZ}. For more background on the numerical evaluation of scale functions, we refer to \cite{SUR}.

We analyze in Subsection \ref{ssec: workload dist levy} the distribution of the workload at the exponentially distributed time $T_\beta$ for an arbitrary initial workload, and in Subsection \ref{sec: prelim overshoot transform} the overshoot of the free process $Y$ over level $u\geqslant0$.

\subsection{Workload distribution}\label{ssec: workload dist levy}
This subsection studies the probability density
\begin{equation}\label{eq:cdf}
    \mathbb{P}_x(V(T_\beta) \in \mathrm{d}y) := \mathbb{P}(V(T_\beta) \in \mathrm{d}y\,|\,V(0) = x),
\end{equation}
for $x\in[0,K]$ and $y\in[0,K]$, and the corresponding LST
\begin{align}\label{eq:lst}
    \mathbb{E}_x \left(e^{-\alpha V(T_\beta)}\right) &:= \mathbb{E} \left(e^{-\alpha V(T_\beta)}\,\middle| \,V(0) = x\right)
\end{align}
for $x\in[0,K]$, $\alpha\geqslant 0$ and $\beta>0$. We distinguish between the cases where the L\'evy process $Y$ is a subordinator (i.e., non-decreasing) and where it is not. 

\subsubsection*{Non-subordinator case}\label{sec: prelim non-sub}
We first consider the case in which $Y$ is {\it not} a subordinator — for example, a compound Poisson process with negative drift between jumps or a Brownian motion. The hitting probabilities in this setting are well known from the literature and admit the following simple representations in terms of the scale functions.
\begin{lemma}\label{lemma: deltas Lévy non subord}
    For $u_-, u_+ \geqslant 0$ with $u_-+u_+ > 0$ and $\beta>0$,
    \begin{align*}
        \delta_{-}(u_-, u_+, \beta) &= \frac{W^{(\beta)}(u_+)}{W^{(\beta)}(u_- + u_+)}, \\
        \delta_{+}(u_-, u_+, \beta) &= Z^{(\beta)}(u_+) -  Z^{(\beta)}(u_-+ u_+) \frac{W^{(\beta)}(u_+)}{W^{(\beta)}(u_- + u_+)}.
    \end{align*}
\end{lemma}
\begin{proof}
    Follows immediately from \cite[Theorem 8.1]{KYP}; note that the result there is for spectrally-negative L\'evy processes not necessarily starting at~$0$, so minor adjustments are required.
\end{proof}

The following lemma states the main result of \cite{IP} and gives an expression for $\mathbb{P}_x(V(T_\beta) \in \mathrm{d}y)$ in terms of the scale functions.

\begin{lemma}\label{lemma: dist W init x}
    {\em (i)}~For $x\in[0,K]$, $y\in(0,K]$ and $\beta>0$,
    \begin{equation}\label{eq: dens W init x Levy}
        \mathbb{P}_x\left(V(T_\beta) \in \mathrm{d}y\right) = \left(\frac{Z^{(\beta)}(K-x)}{W^{(\beta)}(K)}\frac{\mathrm{d}}{\mathrm{d}y}W^{(\beta)}(y) - W^{(\beta)}(y-x)\right) \mathrm{d}y.
    \end{equation}
    {\em (ii)}~For $x\in[0,K]$ and $\beta>0$,
    \begin{equation}\label{eq: empty prob W init x Levy}
        \mathbb{P}_x\left(V(T_\beta) = 0\right) = Z^{(\beta)}(K-x)\frac{W^{(\beta)}(0)}{W^{(\beta)}(K)}.
    \end{equation}
\end{lemma}

Using the expressions in Lemma \ref{lemma: dist W init x}, the following lemma derives the LST \eqref{eq:lst} in terms of scale functions. The proof can be found in Appendix \ref{app: proofs lemmas}.

\begin{lemma}\label{lemma: LST W Levy inti x}
    For $\alpha \geqslant 0$ and $\beta >0$,
    \begin{equation}\label{eq: LST Levy W init x}
        \mathbb{E}_x\left(e^{-\alpha V(T_\beta)}\right) = \left(e^{-\alpha K} + \int_0^K \alpha e^{-\alpha y}\frac{W^{(\beta)}(y)}{W^{(\beta)}(K)}\,\mathrm{d}y\right)Z^{(\beta)}(K-x) - \beta e^{-\alpha x}\int_0^{K-x}e^{-\alpha y}W^{(\beta)}(y)\,\mathrm{d}y.
    \end{equation}
\end{lemma}

\begin{remark}
    {\em A key result for queues with spectrally-positive L\'evy input is what could be termed the `time-dependent version of the generalized Pollaczek-Khinchine formula' (see e.g.\ \cite[Theorem 4.1]{DM}). This result states that, in the setting of {\it infinite} capacity, the LST of the workload at time $T_\beta$ is given by
     \begin{equation}\label{eq: PK}\frac{\beta}{\beta-\varphi(\alpha)}\left(e^{-\alpha x} -\frac{\alpha}{\psi(\beta)}e^{-\psi(\beta) \,x}\right),\end{equation}
     with as before $x$ the initial workload level, $\varphi(\cdot)$ the Laplace exponent of $Y$, and $\psi(\cdot)$ its right inverse. Importantly, the expression \eqref{eq: LST Levy W init x} extends this classical result to the finite-capacity setting.
     
    It is readily verified that \eqref{eq: PK} is consistent with \eqref{eq: LST Levy W init x} by taking the limit $K\to\infty$ in \eqref{eq: LST Levy W init x}. Indeed, from the definition of the scale function,
    \[
    -\lim_{K\to\infty}  \beta e^{-\alpha x}\int_0^{K-x}e^{-\alpha y}W^{(\beta)}(y)\,\mathrm{d}y = \frac{\beta}{\beta-\varphi(\alpha)}e^{-\alpha x},
    \] 
    In addition,
    \begin{align*}
        \lim_{K\to\infty}  &\left(e^{-\alpha K} + \int_0^K \alpha e^{-\alpha y}\frac{W^{(\beta)}(y)}{W^{(\beta)}(K)}\,\mathrm{d}y\right)Z^{(\beta)}(K-x) \\&=\lim_{K\to\infty}  \left(
        \int_0^\infty \alpha e^{-\alpha y}\frac{W^{(\beta)}(\min(y,K))}{W^{(\beta)}(K)}\,\mathrm{d}y
        \right)Z^{(\beta)}(K-x)\\
        &=\alpha \int_0^\infty e^{-\alpha y}\,W^{(\beta)}(y)\,{\rm d}y \lim_{K\to\infty}\frac{Z^{(\beta)}(K-x) }{W^\beta(K)}=-\frac{\alpha}{\beta-\varphi(\alpha)}\lim_{K\to\infty}\frac{Z^{(\beta)}(K-x) }{W^\beta(K)}.
    \end{align*}
    By \cite[Lemma 3.3]{KUZ}, $W^{(\beta)}(y)e^{-\psi(\beta)\,y}\to 1/\varphi'(\psi(\beta))$ and $Z^{(\beta)}(y)/W^{(\beta)}(y)\to \beta/\psi(\beta),$
    so that
    \[\lim_{K\to\infty}\frac{Z^{(\beta)}(K-x) }{W^\beta(K)} =\lim_{K\to\infty}\frac{Z^{(\beta)}(K-x) }{W^\beta(K-x)}\frac{W^{(\beta)}(K-x) }{W^\beta(K)} = \frac{\beta}{\psi(\beta)} e^{-\psi(\beta)\,x}.\]
    Adding up the two contributions, we readily obtain \eqref{eq: PK}. \hfill$\Diamond$
    }
\end{remark}

\subsubsection*{Subordinator case}\label{sec: prelim sub}
For the remainder of this subsection, we assume $Y$ is a subordinator, i.e., almost surely non-decreasing. Examples include the compound Poisson process with non-negative drift and the gamma process. For details on subordinators, see \cite[\S 2.6.2]{KYP}.
Again, our goal is to characterize the distribution of $V(T_\beta)$ when starting at the boundaries $0$ or $K$.
We note the subtlety that $Y(T_{\beta})$ may now have one or more atoms, as is the case, for example, with a compound Poisson process with no drift and discrete jumps.

\medskip

The following two lemmas are straightforward and included for completeness. Note that  since $Y$ is non-decreasing, we have $\sigma(u_-)=\infty$ almost surely for any $u_-,u_+\geqslant 0$ and $\beta>0$. As a consequence, $\delta_-(u_-,u_+,\beta)\equiv 0$ and $\delta_+(u_-,u_+,\beta)=\mathbb{P}(\tau(u_+)\leqslant T_\beta)$.

\begin{lemma}\label{lemma: dist W Levy sub init x}
    {\rm (i)}~For $y\in[0,K)$ and $\beta>0$,
    \begin{equation*}
        \mathbb{P}_x\left(V(T_\beta) \leqslant y\right) = \mathbb{P}\left(Y(T_\beta) \leqslant y-x\right).
    \end{equation*}
    {\rm (ii)}~For $\beta>0$,  \begin{align*}
        \mathbb{P}_x(V(T_\beta) = K) = \mathbb{P}\left(Y(T_\beta) \geqslant K-x\right).
    \end{align*}
\end{lemma}
\begin{lemma}\label{lemma: LST sub init x}
    For $\alpha\geqslant0$ and $\beta >0$,
    \begin{equation*}
        \mathbb{E}_x\left(e^{-\alpha V(T_\beta)}\right) = e^{-\alpha K}\mathbb{P}\left(Y(T_\beta) \geqslant K-x\right) + e^{-\alpha x}\int_{[0,K-x)} e^{-\alpha y}\,\mathbb{P}\left(Y(T_\beta) \in \mathrm{d}y\right).
    \end{equation*}
\end{lemma}


\subsection{Overshoot transform}\label{sec: prelim overshoot transform}
We next consider the overshoot of the L\'evy process $Y$ over a level $u \geqslant 0$ and analyze its LST on the event it crosses $u$ before killing:
\begin{equation*}
    \eta(u, \alpha, \beta) := \mathbb{E}\left(e^{-\alpha(Y(\tau(u)) - u)}{\bs 1}_{\{\tau(u) \leqslant T_\beta\}}\right)
\end{equation*}
for $\alpha\geqslant0, \beta>0$ and $u\geqslant0$. We analyze $\eta(u,\alpha,\beta)$ by evaluating the associated Laplace transform with respect to $u$: for $\gamma\geqslant0$,
\begin{equation}
    \zeta(\alpha, \beta, \gamma) := \int_0^\infty e^{-\gamma u}\eta(u,\alpha, \beta)\,\mathrm{d}u.
\end{equation}
Let $\bar{Y}(t) := \max\{Y(s): s\in[0,t]\}$ denote the running maximum of $Y$ until time $t$. For $\alpha\geqslant0$ and $\beta>0$, let $\mathscr{Y}(\alpha, \beta)$ denote the LST of $\bar{Y}(T_\beta)$. By e.g.\ \cite[Section 16.3]{DM},
\begin{equation}\label{eq: zeta in terms of LST}
    \zeta(\alpha, \beta, \gamma) = \frac{\mathscr{Y}(\gamma,\beta)}{\gamma-\alpha}\left(\frac{1}{\mathscr{Y}(\gamma,\beta)}-\frac{1}{\mathscr{Y}(\alpha,\beta)}\right).
\end{equation}
In the following two lemmas we present closed-form expressions for $\zeta(\alpha,\beta,\gamma)$, distinguishing between subordinators and non-subordinators.
\begin{lemma}\label{lemma: eta and zeta Levy}
    Suppose $Y$ is not a subordinator. For $\alpha, \gamma \geqslant 0$, $\beta >0$ and $u\geqslant0$,
    \begin{align}\label{eq: zeta Levy nonsub}
        \zeta(\alpha, \beta, \gamma) &= \frac{1}{\beta - \varphi(\gamma)} \left( \frac{\varphi(\alpha) - \varphi(\gamma)}{\gamma - \alpha} - \frac{\varphi(\alpha) - \beta}{\psi(\beta) - \alpha} \right),\\
    \label{eq: eta Levy nonsub}
        \eta(u,\alpha,\beta)& = e^{\alpha u}\left(1-(\varphi(\alpha) - \beta)\int_0^u e^{-\alpha y}W^{(\beta)}(y)\,\mathrm{d}y\right) + \frac{\varphi(\alpha) - \beta}{\psi(\beta) - \alpha} W^{(\beta)}(u). 
    \end{align}
\end{lemma}
\begin{proof}
    By \cite[Equation (15.5)]{DM},
    \begin{equation*}
        \mathscr{Y}(\alpha,\beta) = \frac{\psi(\beta) - \alpha}{\beta - \varphi(\alpha)}\frac{\beta}{\psi(\beta)},
    \end{equation*}
  Substituting this into \eqref{eq: zeta in terms of LST} and simplifying yields \eqref{eq: zeta Levy nonsub}. The validity of \eqref{eq: eta Levy nonsub} follows by verifying that its Laplace transform with respect to $u$ equals \eqref{eq: zeta Levy nonsub}. Since $\eta(u,\alpha,\beta)$ is continuous for all $u \in [0,\infty)$ in the non-subordinator case, it is uniquely determined by $\zeta(\alpha,\beta,\gamma)$.
\end{proof}

\begin{lemma}\label{lemma: eta and zeta Levy subord}
    Suppose $Y$ is a subordinator. For $\alpha, \gamma \geqslant0$, $\beta>0$ and $u\geqslant0$,
    \begin{align}\label{eq: zeta Levy sub}
        \zeta(\alpha, \beta, \gamma) &= \frac{1}{\beta - \varphi(\gamma)} \frac{\varphi(\alpha) - \varphi(\gamma)}{\gamma - \alpha},\\
    \label{eq: eta Levy sub}
        \eta(u,\alpha,\beta) &= \frac{\beta - \varphi(\alpha)}{\beta}e^{\alpha u}\int_{(u,\infty)} e^{-\alpha y}\mathbb{P}(Y(T_\beta)\in\mathrm{d}y).
    \end{align}
\end{lemma}
\begin{proof}
    We now have $Y(T_\beta) = \bar{Y}(T_\beta)$, so that trivially $ \mathscr{Y}(\alpha,\beta) = {\beta}/({\beta - \varphi(\alpha)})$.
    It is easy to verify that substituting this LST into \eqref{eq: zeta in terms of LST}  yields \eqref{eq: zeta Levy sub}. Transforming \eqref{eq: eta Levy sub} with respect to $u$ yields \eqref{eq: zeta Levy sub}. Here it should be noted that transforming a similar expression as in \eqref{eq: eta Levy sub}, but integrating the probability measure over $[u,\infty)$ instead, yields \eqref{eq: zeta Levy sub} as well. However, since $\tau(u)$ denotes the first time $Y$ \textit{strictly} exceeds the level $u$, $\eta(u,\alpha,\beta)$ has to be right-continuous in all $u\in[0,\infty)$ and therefore \eqref{eq: eta Levy sub} is the correct characterization of $\eta(u,\alpha,\beta)$.
\end{proof}

\section{Analysis}\label{sec: analysis}
In this section we present our main results. We express the transform $\chi_{ij}(x,\alpha,\beta)$ in terms of a number of auxiliary objects, which we subsequently identify in Section \ref{sec: aux}. 
We follow a similar line of reasoning as in \cite[Exercise 5.5]{MB}, noting that there are various additional challenges arising from the fact that in this paper we work with a general dimension $d\in{\mathbb N}$ and spectrally-positive Lévy processes, including subordinators. We also remark in passing that the analysis in \cite[Exercise 5.5]{MB} contains a minor, easily correctable flaw, which does not affect the overall approach.

\begin{remark}\label{remark: state dependent killing}
{\em  It turns out that our framework can be extended relatively easily to allow for \emph{state-dependent killing}; see, for example, \cite[\S 3.1]{MB}. In this setting, we work with a killing time $T_{\bs\beta}$, where $\bs\beta$ is now a vector, with $\beta_i>0$ for all $i\in\defD $. 
Conditional on the background process being in state $i$, $T_{\bs\beta}$ is exponentially distributed with parameter $\beta_i$.}
\hfill$\Diamond$
\end{remark}

In what follows, all matrices are of dimension $d\times d$, unless stated otherwise. E.g., we define the matrix $\bs\chi(x)$ which has entries $\chi_{ij}(x) \equiv \chi_{ij}(x,\alpha,\beta),\ i,j\in\defD$.

\subsection{Decomposition}\label{sec: decomposition}
We condition  on the first `event': this can be (i)~hitting zero, (ii)~hitting~$K$, or (iii)~the killing at time $T_{\beta}$. We therefore define two $d\times d$ matrices, $\bs{\delta}_-(u_-, u_+, \beta)$ and $\bs{\delta}_+(u_-, u_+, \beta)$, as matrix counterparts of the hitting probabilities defined at the beginning of Section~\ref{sec: prelim}. Their entries are defined, for $\beta > 0$, $u_- \geqslant0$ and $u_+\geqslant0$ with $u_-+u_+ > 0$, as
\begin{align*}
    \delta_{-,ij}(u_-, u_+, \beta) &:= \mathbb{P}_{i}\left(\sigma(u_-) \leqslant \min\{\tau(u_+), T_{\beta}\}, J(\sigma(u_-)) = j\right) \\
    \delta_{+,ij}(u_-, u_+, \beta) &:= \mathbb{P}_{i}\left(\tau(u_+) \leqslant \min\{\sigma (u_-), T_{\beta}\}, J(\tau(u_+)) = j\right),
\end{align*}
with $\sigma(u_-)$ and $\tau(u_+)$ as defined in Section \ref{sec: prelim}. In the sequel we  write $\delta_{-,ij}(x) \equiv \delta_{-,ij}(x, K-x, \beta)$ and $\delta_{+,ij}(x) \equiv \delta_{+,ij}(x, K-x, \beta)$ for $x\in[0,K]$ to shorten the notation (which is possible, as we keep $K$ and $\beta$ fixed). For now we assume that we have access to these hitting probabilities $\delta_{-,ij}(x)$ and $\delta_{+,ij}(x)$; in Section \ref{sec: aux} we point out how they can be evaluated.

In addition, for $x\in[0,K]$, we define the matrix
$\bs\delta_{\star}(x) $ 
whose entries are given by the LSTs
\begin{equation}
    \delta_{\star ,ij}(x) \equiv \delta_{\star, ij}(x,\alpha, \beta) := e^{-\alpha x}\,\mathbb{E}_i\left( e^{-\alpha Y(T_{\beta})}{\bs 1}_{\{T_{\beta} \leqslant \min\{\sigma(x), \tau(K-x)\},\,J(T_{\beta}) = j\}} \right).\label{eq: def delta star}
\end{equation}

\begin{lemma}\label{lemma: decomp chi(x)}
    For $\alpha \geqslant0, \beta>0$ and $x\in[0,K]$,
    \begin{equation}\label{eq: chi(x)}
        \bs\chi(x) = \bs\delta_{-}(x)\,\bs\chi(0) + \bs\delta_{+}(x) \,\bs\chi(K)+ \bs\delta_{\star}(x).
    \end{equation}
\end{lemma}
\begin{proof}
   By the strong Markov property we obtain
    \begin{equation*}\label{eq: chi_ij(x)}
        \chi_{ij}(x) = \sum_{k=1}^d \delta_{-,ik}(x)\,\chi_{kj}(0) + \sum_{k=1}^d \delta_{+,ik}(x) \,\chi_{kj}(K) + \delta_{\star ,ij}(x).
    \end{equation*}
    The first term corresponds to hitting level $0$ first, the second to hitting level $K$ first, and $\delta_{\star,ij}(x)$ to killing occurring before hitting $0$ or $K$. Expressing this decomposition in matrix form yields $\eqref{eq: chi(x)}$.
    For illustration, Figure~\ref{fig: path decomposition chi(x)} depicts sample paths corresponding to the three scenarios used in the decomposition of ${\bs \chi}(x)$ in \eqref{eq: chi(x)}.
\end{proof}

\begin{figure}[h]
    \begin{subfigure}{0.32\textwidth}
        \includegraphics[scale=0.45]{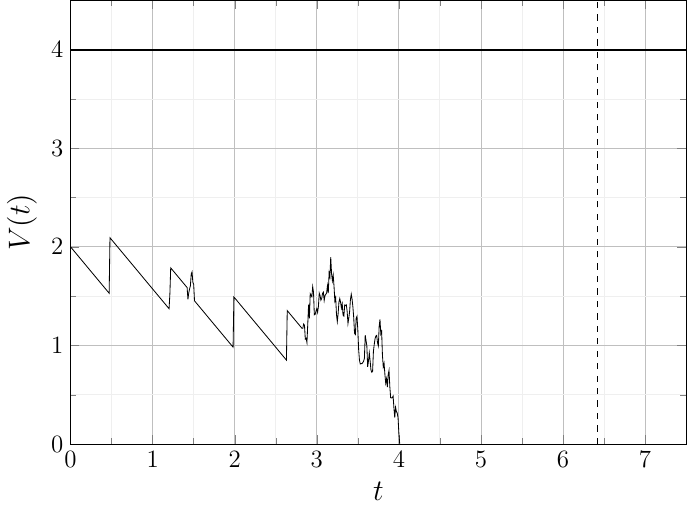}
        \caption{\label{fig: Path 1}}
    \end{subfigure}
   \begin{subfigure}{0.32\textwidth}
        \includegraphics[scale=0.45]{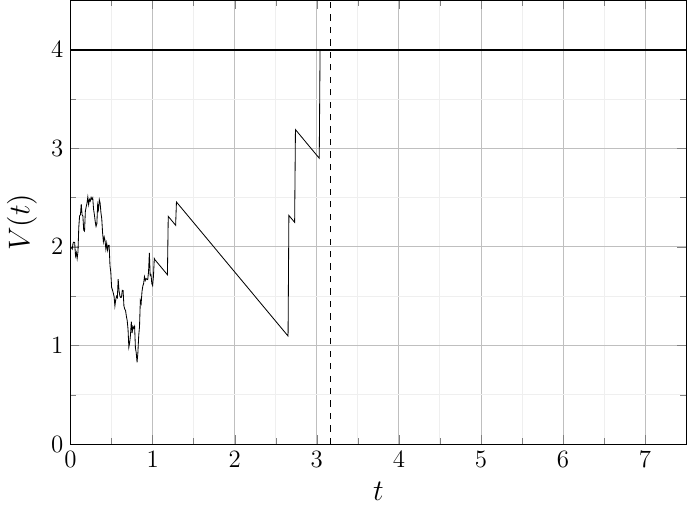}
        \caption{\label{fig: Path 2}}
    \end{subfigure}
    \begin{subfigure}{0.32\textwidth}
        \includegraphics[scale=0.45]{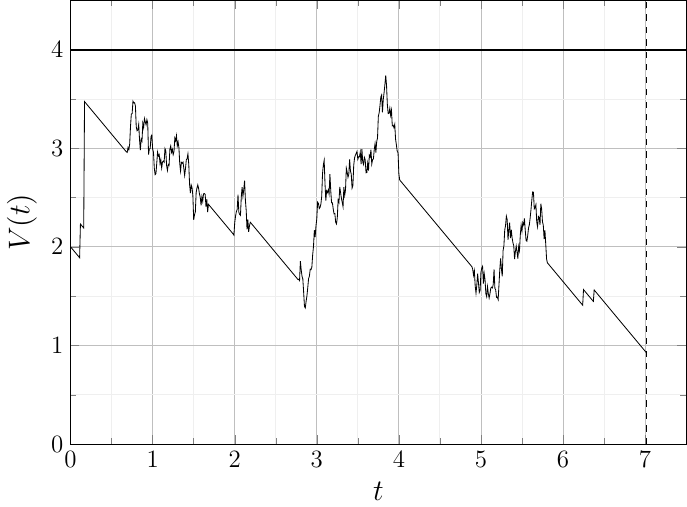}
        \caption{\label{fig: Path 3}}
    \end{subfigure}
    \caption{\label{fig: path decomposition chi(x)}The three scenarios in the decomposition in the proof of Lemma \ref{lemma: decomp chi(x)}, with $K=4$ and $x=2$: ({\sc a}) hitting $0$ first, ({\sc b})~hitting $K$ first, and ({\sc c})~killing first. The vertical dashed line corresponds to a sample of the killing time $T_\beta$.}
\end{figure}

As an aside, we note that the decomposition in \eqref{eq: chi(x)} may be trivial when $x \in \{0,K\}$, but this is not always the case. For a Markov-modulated Brownian motion and $x=0$, for example, we have that ${\bs \delta}_-(0)$ is simply the $d\times d$ identity matrix ${\bs I}$, so that \eqref{eq: chi(x)} reads ${\bs \chi}(0) = {\bs \chi}(0)$;  an analogous observation holds for $x=K$ and ${\bs \delta}_+(K)$. On the other hand, for a Markov-modulated compound Poisson process with negative drifts, while ${\bs \delta}_-(0)$ does equal ${\bs I}$, ${\bs \delta}_+(K)$ does not, and \eqref{eq: chi(x)} therefore yields a nontrivial decomposition of ${\bs \chi}(K)$. These observations follow directly from the definitions of the stopping times $\sigma(u_-)$ and $\tau(u_+)$, defined as the first times that $Y$ drops below $u_-$ and strictly exceeds $u_+$, respectively.

\subsection{Finding \texorpdfstring{${\bs \delta}_{\star}(x)$}{delta star}}\label{sec: delta_star}
Bearing in mind how we have rewritten ${\bs \chi}(x)$ in \eqref{eq: chi(x)}, our next objective is to analyze ${\bs \delta}_{\star}(x)$. 
We define the matrix ${\bs \Phi}(\alpha,\beta)$ through its entries: for $\alpha \geqslant0, \beta>0$,
\begin{align*}\label{eq: Phi_ij}
        \Phi_{ij}(\alpha, \beta) &:= \mathbb{E}_i\left(e^{-\alpha Y(T_{\beta})}{\bs 1}_{\{J(T_{\beta}) = j\}}\right),
\end{align*}
recalling that $Y$ is the free (i.e., non-reflected) process.
In addition, with the Laplace exponent $\varphi_i(\alpha):=\log\,\mathbb{E}e^{-\alpha Y_i(1)}$, we let the $(i,j)$-th entry of the matrix $\bs F(\alpha)$ be given by
\begin{equation}\label{eq: defF}
    F_{ij}(\alpha) := q_{ij}\mathscr{B}_{ij}(\alpha) + \varphi_{i}(\alpha)\,\ind{i=j},
\end{equation}
for $i,j\in\defD $, 
with ${\mathscr B}_{ii}(\alpha)= 1$ for all $\alpha\geqslant 0$ and $i\in\defD $, so that
\[\Phi(\alpha, \beta) = \beta \int_0^\infty e^{-\beta \bs{I} t}e^{\bs F(\alpha)\,t}\,{\rm d}t = \beta \left(\beta \bs{I} - \bs F(\alpha)\right)^{-1}.\]

Furthermore, observing that $Y(\tau(u)) - u$ denotes the overshoot over the level $u$, we define the matrix $\bs\eta(u,\alpha,\beta)$ via, for $\alpha\geqslant0, \beta > 0$ and $u\geqslant0$,
\begin{equation}\label{eq: eta definition}
    \eta_{ij} (u,\alpha,\beta):= \mathbb{E}_i\left(e^{-\alpha(Y(\tau(u)) - u)}\bs{1}_{\{\tau(u) \leqslant T_{\beta},\,J(\tau(u)) = j\}}\right).
\end{equation}
This quantity is the LST of the overshoot, on the event that level $u$ is exceeded before killing while the background process is in state $j$, conditional on starting in state $i$. It can be viewed as the matrix analog of the overshoot LST $\eta(u,\alpha,\beta)$ for Lévy processes studied in Section \ref{sec: prelim overshoot transform}.
For now, we assume that ${\bs \eta}(u,\alpha,\beta)$ is available; it will be uniquely characterized in Section \ref{sec: aux}. If level $u$ is exceeded by a jump occurring at the instant the background process transitions from some $k \neq j$ to $j$, we set $J(\tau(u)) = j$.

\begin{lemma}\label{lemma: delta*} For $\alpha \geqslant0, \beta>0$ and $x\in[0,K]$,
    \begin{equation}\label{eq: delta*}
        \bs\delta_\star(x) = e^{-\alpha x}\left( \bs{I} - e^{\alpha x}\bs\delta_-(x) - e^{-\alpha(K-x)}\bs\eta(K-x,\alpha,\beta) + e^{-\alpha(K-x)}\bs\delta_-(x)\,\bs\eta(K,\alpha,\beta) \right)\bs\Phi(\alpha,\beta).
    \end{equation}
\end{lemma}
\begin{proof} 
    We start by decomposing $\bs\delta_\star(x)$. It is readily verified that
   \begin{equation}\label{eq: decomp delta*}
       \bs\delta_\star(x) = e^{-\alpha x}\left(\bs\Phi(\alpha,\beta) - {\bs \Xi}(x,\alpha, \beta) -{\bs \Psi}(x,\alpha, \beta)\right),
   \end{equation}
    where the matrices ${\bs \Xi}(x,\alpha,\beta)$ and ${\bs \Psi}(x,\alpha,\beta)$ have entries
\begin{align*}
    \Xi_{ij}(x,\alpha,\beta)&:= {\mathbb E}_i\left(e^{-\alpha Y(T_{\beta})}{\bs 1}_{\{\sigma(x)\leqslant\min\{\tau(K-x), T_{\beta} \},\,J(T_{\beta})=j\}}\right),\\
    \Psi_{ij}(x,\alpha,\beta)&:= {\mathbb E}_i\left(e^{-\alpha Y(T_{\beta})}{\bs 1}_{\{\tau(K-x)\leqslant\min\{\sigma(x), T_{\beta} \},\,J(T_{\beta})=j\}}\right).
\end{align*}Hence, to analyze $\bs\delta_\star(x)$, we may consider all sample paths of the free process $Y$ except those on which the level $-x$ or $K-x$ is reached first, before $T_\beta$. We proceed by analyzing (i)~${\bs \Xi}(x,\alpha,\beta)$ and (ii)~${\bs \Psi}(x,\alpha,\beta)$.

\medskip
(i)~By the strong Markov property, we find that the entries of ${\bs \Xi}(x,\alpha,\beta)$ are simply given by
\begin{align*}
    \Xi_{ij}(x,\alpha,\beta) = e^{\alpha x}\sum_{k=1}^d \delta_{-,ik}(x)\,{\Phi}_{kj}(\alpha, \beta).
\end{align*}
The reasoning for this decomposition is that in this scenario the level $-x$ is reached (with equality, due to the absence of negative jumps) before $T_{\beta}$ and $\tau(K-x)$, where the summation over $k$ takes care of the possible states of the background process at time $\sigma(x)$.

\medskip
(ii)~We further decompose ${\bs \Psi}(x,\alpha,\beta)$ by considering all paths in which the level $K-x$ is reached before $T_{\beta}$ except those in which the level $-x$ is reached before the level $K-x$. This gives
\begin{equation}\label{eq: decomp Psi}
    {\Psi}_{ij}(x,\alpha,\beta) = {\mathbb E}_i\left(e^{-\alpha Y(T_{\beta})}{\bs 1}_{\{\tau(K-x)\leqslant T_{\beta},\,J(T_{\beta})=j\}}\right) - {\mathbb E}_i\left(e^{-\alpha Y(T_{\beta})}{\bs 1}_{\{\sigma(x)\leqslant\tau(K-x)\leqslant T_{\beta},\,J(T_{\beta})=j\}}\right).
\end{equation}
The first component in the right-hand side of \eqref{eq: decomp Psi} can be rewritten as 
\begin{equation*}
    {\mathbb E}_i\left(e^{-\alpha Y(T_{\beta})}{\bs 1}_{\{\tau(K-x)\leqslant T_{\beta},\,J(T_{\beta})=j\}}\right) = e^{-\alpha(K-x)}\sum_{k=1}^d  \eta_{ik}(K-x, \alpha, \beta)\,\Phi_{kj}(\alpha,\beta),
\end{equation*}
since in this case the level $K-x$ is reached before $T_{\beta}$, with $k$ and $\eta_{ik}(K-x,\alpha,\beta)$ accounting for the background state and the overshoot at time $\tau(K-x)$, respectively.

The second component in the right-hand side of \eqref{eq: decomp Psi} can be rewritten as
\begin{equation*}
        {\mathbb E}_i\left(e^{-\alpha Y(T_{\beta})}{\bs 1}_{\{\sigma(x)\leqslant\tau(K-x)\leqslant T_{\beta},\,J(T_{\beta})=j\}}\right) = \sum_{k=1}^d\sum_{\ell=1}^d\delta_{-,ik}(x)\,e^{-\alpha(K-x)}\eta_{k\ell}(K,\alpha,\beta)\,\Phi_{\ell j}(\alpha,\beta),
    \end{equation*}
where we argue that, in this case, the paths first reach level $-x$ (with equality) and subsequently reach level $K-x$ (or, equivalently, level $K$ in a system shifted by $x$ and starting at $0$) before $T_{\beta}$.

\medskip

Inserting the found expressions for ${\bs \Xi}(x,\alpha,\beta)$ and ${\bs \Psi}(x,\alpha,\beta)$ into \eqref{eq: decomp delta*}, we obtain the following matrix equation for $\bs{\Psi}(x,\alpha,\beta)$:
\begin{equation*}
    {\bs \Psi}(x,\alpha,\beta) = e^{-\alpha(K-x)}\left(\bs{\eta}(K-x, \alpha, \beta) - \bs{\delta}_{-}(x)\,\bs{\eta}(K,\alpha,\beta)\right)\,\bs{\Phi}(\alpha,\beta)
\end{equation*}
Applying the identity \eqref{eq: decomp delta*}, we can now determine $\bs\delta_{\star}(x)$ in a straightforward manner.
\end{proof}

\subsection{Finding \texorpdfstring{${\bs \chi}(0)$}{chi 0} and \texorpdfstring{${\bs \chi}(K)$}{chi K}}\label{sec: chi(0) and chi(K)}
Returning to the identity \eqref{eq: chi(x)}, and recalling that the evaluation of ${\bs \delta}_-(x)$ and ${\bs \delta}_+(x)$ is addressed in Section \ref{sec: aux},
the last step is to identify the `constant matrices' $\bs\chi(0)$ and $\bs\chi(K)$.  Below we point out that  we can evaluate the $2d^2$ entries by solving $d$ systems of linear equations, each of them having $2d$ unknowns. 

\medskip

Let $V_i \equiv (V_i(t))_{t\geqslant 0}$ denote a workload process driven solely by the Lévy process $Y_i$. Define the vectors $\bs{q} \equiv (q_1,\dots,q_d)$ and $\bs{\omega}$, with entries $\omega_i := \beta + q_i$, and define the $d\times d$ matrices $\bs{P}^{\{0,0\}}$, $\bs{P}^{\{0,K\}}$, $\bs{P}^{\{K,0\}}$, and $\bs{P}^{\{K,K\}}$ through their entries as follows. For $i,\ell \in \defD $ and $x\in\{0,K\}$,
\begin{align*}
    P_{i\ell}^{\{x,0\}} &:= \sum_{k=1,k\neq i}^d \frac{q_{ik}}{\omega_i}\int_{[0,K]}\mathbb{P}_x\left(\min\{V_i(T_{\omega_i}) + B_{ik}, K\}\in\mathrm{d}y\right)\,\delta_{-,k\ell}(y) \\
    P_{i\ell}^{\{x,K\}} &:= \sum_{k=1,k\neq i}^d \frac{q_{ik}}{\omega_i}\int_{[0,K]}\mathbb{P}_x\left(\min\{V_i(T_{\omega_i}) + B_{ik}, K\}\in\mathrm{d}y\right)\,\delta_{+,k\ell}(y) .
\end{align*}
Also, the matrices ${\bs b}^{\{0\}}$ and ${\bs b}^{\{K\}}$ are defined through their entries: for $i,\ell \in \defD $ and $x\in\{0,K\}$,
\begin{align*}
    {b}^{\{x\}}_{i\ell} &:= \bs{1}_{\{i=\ell\}}\frac{\beta}{\omega_i}\mathbb{E}_x\left(e^{-\alpha V_i(T_{\omega_i})}\right) + \sum_{k=1, k\neq i}^d \frac{q_{ik}}{\omega_i}\int_{[0,K]}\mathbb{P}_x\left(\min\{V_i(T_{\omega_i}) + B_{ik}, K\}\in\mathrm{d}y\right)\,\delta_{\star,k\ell}(y).
\end{align*}
{The six matrices $\bs{P}^{\{x,y\}}$, and ${\bs b}^{\{x\}}$, with $x,y\in\{0,K\}$, defined above are given in terms of the distribution and LST of $V_i(T_{\omega_i})$ initialized in either of its boundaries (i.e., $0$ or $K$). Note that these were given in Section~\ref{ssec: workload dist levy}: see Lemmas \ref{lemma: dist W init x} and~\ref{lemma: dist W Levy sub init x} for the distribution functions, and Lemmas \ref{lemma: LST W Levy inti x} and \ref{lemma: LST sub init x} for the corresponding LSTs.} 
The matrices  also involve entries of ${\bs \delta}_\star(y)$ from Lemma \ref{lemma: delta*}, and ${\bs \delta}_-(y)$ and ${\bs \delta}_+(y)$, which are characterized in Section \ref{sec: hitting probabilities}. The density $\mathbb{P}_x\left(\min\{V_i(T_{\omega_i}) + B_{ik}, K\}\in\mathrm{d}y\right)$ coincides with the density of the workload immediately after the first background-state jump, conditional on this jump being from $i$ to $k$ and occurring before killing. See also Remark \ref{rem: chis} below.

\begin{lemma}\label{lemma: decomp chi(0) and chi(K)}
    For $\alpha\geqslant0$ and $\beta>0$,
    \begin{equation*}
        \left(\begin{array}{c} {\bs \chi}(0) \\ {\bs \chi}(K)\end{array} \right)= {\bs P}\left(\begin{array}{c} {\bs \chi}(0) \\ {\bs \chi}(K)\end{array} \right) + 
        \left(\begin{array}{c} {\bs b}^{\{0\}} \\ {\bs b}^{\{K\}}\end{array} \right),
    \end{equation*}
    with ${\bs P}$ the $2d\times 2d$ matrix given by
    \begin{equation*}
        {\bs P}:=\left(\begin{array}{cc} {\bs P}^{\{0,0\}}&{\bs P}^{\{0,K\}}\\{\bs P}^{\{K,0\}}&{\bs P}^{\{K,K\}}\end{array}\right).
    \end{equation*}
\end{lemma}
\begin{proof}
    The key step in the proof is, again, conditioning on the first event, which in this case is either killing or a jump of the background process. If killing occurs first, ${\bs \chi}(0)$ and ${\bs \chi}(K)$ depend only on the LST of $V_i(T_{\omega_i})$, as captured by the matrices ${\bs b}^{\{0\}}$ and ${\bs b}^{\{K\}}$. If instead a background-state jump occurs first, we take the post-jump workload as the new initial value and use the decomposition of ${\bs \chi}(x)$ from Lemma \ref{lemma: decomp chi(x)} to express ${\bs \chi}(x)$ in terms of ${\bs \chi}(0)$ and ${\bs \chi}(K)$. A detailed proof of the resulting matrix equation is given in Appendix \ref{app: proofs lemmas}.
\end{proof}
\begin{lemma}\label{lemma: diagdom}
    For any $\beta>0$, the matrix ${\bs I}-{\bs P}$ is strictly diagonally dominant.
\end{lemma}
\begin{proof}
    We are to show that $(\bs{P}^{\{0,0\}}+\bs{P}^{\{0,K\}}){\bs 1}< {\bs 1}$ as well as $(\bs{P}^{\{K,0\}}+\bs{P}^{\{K,K\}}){\bs 1}< {\bs 1}$, where these inequalities are meant in the component-wise sense. We focus on the former claim, as the latter claim is proven analogously. Observe that
     \begin{align*}
        \sum_{\ell=1}^d \bigl(P^{\{0,0\}}_{i\ell}+P^{\{0,K\}}_{i\ell}\bigr)
        = \sum_{\ell=1}^d \sum_{k=1,k\neq i}^d \frac{q_{ik}}{\omega_i}
        \int_{[0,K]} \mathbb{P}_0\!\left(\min\{V_i(T_{\omega_i})+B_{ik},K\}\in{\rm  d}y\right)
        \bigl(\delta_{-,k\ell}(y)+\delta_{+,k\ell}(y)\bigr)
    \end{align*}
     for any $i\in\defD $.
    Observe that $\sum_{\ell=1}^d \delta_{-,k\ell}(y) + \delta_{+,k\ell}(y)\leqslant 1$ for any $k\in\defD$ and any $y\in[0,K]$. Hence,
    \begin{align*}
         \sum_{\ell=1}^d \bigl(P^{\{0,0\}}_{i\ell}+P^{\{0,K\}}_{i\ell} \bigr)&\leqslant  \sum_{k=1,k\neq i}^d \frac{q_{ik}}{\omega_i}\int_{[0,K]}\mathbb{P}_0\left(\min\{V_i(T_{\omega_i}) + B_{ik}, K\}\in\mathrm{d}y\right)=\frac{q_i}{q_i + \beta} < 1
     \end{align*}
    for any $\beta>0$.
    This proves that $\bs{I}-\bs{P}$ is strictly diagonally dominant. 
\end{proof}
\begin{theorem}\label{thrm: general chi(x)}
    The matrix $\bs\chi(x)$ is given by \eqref{eq: chi(x)}, where $\bs\delta_\star(x)$ follows from Lemma $\ref{lemma: delta*}$, and  \begin{equation*}
        \left(\begin{array}{c} {\bs \chi}(0) \\ {\bs \chi}(K)\end{array} \right)=( {\bs I}-{\bs P})^{-1} \left(\begin{array}{c} {\bs b}^{\{0\}} \\ {\bs b}^{\{K\}}\end{array} \right).
    \end{equation*}
\end{theorem}
 
\begin{proof}
    This theorem follows directly from Lemma \ref{lemma: decomp chi(0) and chi(K)}, with Lemma \ref{lemma: diagdom} implying the invertibility of the matrix ${\bs I}-{\bs P}$.
\end{proof}

\begin{remark}\label{rem: chis}
    {\em Above we indicated how the entries of the matrices ${\bs P}$, ${\bs b}^{\{0\}}$, and ${\bs b}^{\{K\}}$ can be determined. In the special case where $Y$ is a Markov-modulated compound Poisson process with strictly negative drifts, a more straightforward procedure can be followed (see Proposition~\ref{prop: decomp chi(0) CPP} in Section \ref{sec: extensions}).} \hfill$\Diamond$
\end{remark}

\section{Computing auxiliary objects}\label{sec: aux}
To complete the analysis, we determine several objects introduced in the previous section, namely: (i) the overshoot transform $\bs\eta(u,\alpha,\beta)$ via its transform in $u$, and (ii) the hitting probabilities $\bs\delta_{-}(x)$ and $\bs\delta_{+}(x)$ using the \emph{scale matrix} framework.  
We distinguish between background states that are subordinators and those that are not. Let $d_- \leqslant d$ denote the number of non-subordinators and set $d_+ := d-d_-$. Without loss of generality, we order the states so that the non-subordinators come first, i.e., $\defD =\defDmin \cup\defDplus$ with $\defDmin :=\{1,\ldots,d_-\}$ and $\defDplus :=\{d_-+1,\ldots,d\}$.

\subsection{Overshoot transform}\label{sec: overshoot transform} 
The objective of this subsection is to uniquely characterize, for $\alpha,u\geqslant 0$ and $\beta>0$, the overshoot transform matrix $\bs\eta(u,\alpha,\beta)$ defined in \eqref{eq: eta definition}; for $d=1$ this was already done in Section \ref{sec: prelim overshoot transform}.
We analyze $\bs\eta(u,\alpha,\beta)$ by evaluating the associated transform, for $\gamma\geqslant 0$,
\begin{equation*}
    \zeta_{ij}(\alpha,\beta,\gamma) = \int_0^\infty e^{-\gamma u} \,\eta_{ij}(u,\alpha,\beta)\,{\rm d}u.
\end{equation*}
In the following two subsections we distinguish between whether the state $i\in\defD $ is a subordinator state or not and derive $\zeta_{ij}(\alpha, \beta, \gamma)$ for both cases. 

\subsubsection*{Non-subordinator states}
We first focus on states $i\in\defDmin $ and set up a system of equations for the entries of $\bs\eta(u,\alpha,\beta)$. Following the approach of \cite[\S 1.3]{MB}, we condition on the first event (which is either a background-state transition or killing). Applying the {\it Wiener-Hopf decomposition} \cite[Chapter VI]{KYP} then yields the following `master equation':
\begin{align}\label{eq: master equation eta}
    \notag\eta_{ij}(u,\alpha,\beta) =
    \notag\:\:\:\:\: &{\bs 1}_{\{i=j\}}\mathbb{E}\left( e^{-\alpha(Y_i(\tau_i(u)) - u)}\bs{1}_{\{\tau_i(u) \leqslant T_{\omega_i}\}}\right)\,+ \\
    \notag\:\:\:\:\:\, &{\bs 1}_{\{i\neq j\}}\frac{q_{ij}}{\omega_i} \int_0^u  \mathbb{P}\left(\bar{Y}_i(T_{\omega_i})\in\mathrm{d}y\right) \int_{0}^\infty \mathbb{P}\left(\bar{Y}_i(T_{\omega_i})-Y_i(T_{\omega_i})\in\mathrm{d}z\right) \\
    \notag&\:\:\:\:\:\:\:\:\:\:\cdot\int_{u-(y-z)}^\infty \mathbb{P}\left(B_{ij}\in\mathrm{d}\nu\right)\,e^{-\alpha(y-z+\nu-u)}\,+\\
    \:\:\:\:\:& \sum_{k=1,k\neq i}^d \frac{q_{ik}}{\omega_i} \int_0^u  \mathbb{P}\left(\bar{Y}_i(T_{\omega_i})\in\mathrm{d}y\right) \int_{0}^\infty \mathbb{P}\left(\bar{Y}_i(T_{\omega_i})-Y_i(T_{\omega_i})\in\mathrm{d}z\right) \\
    \notag&\:\:\:\:\:\:\:\:\:\:\cdot\int_{0}^{u-(y-z)} \mathbb{P}\left(B_{ik}\in\mathrm{d}\nu\right)\,\eta_{kj}(u-(y-z+v), \alpha,\beta),
\end{align}
where $\bar{Y}_i$ denotes the running maximum process pertaining to $Y_i$. Note that, using that $Y_i$ is spectrally-positive, $\bar{Y}_i(T_{\omega_i})-Y_i(T_{\omega_i})$ is exponentially distributed with parameter $\psi_i(\omega_i)$, where $\psi_i(\cdot)$ denotes the right inverse of $\varphi_i(\cdot)$. Thus, the master equation above becomes
\begin{align*}
    \eta_{ij}&(u,\alpha,\beta) =\\
    &\:\:\:\:\: {\bs 1}_{\{i=j\}}\mathbb{E}\left( e^{-\alpha(Y_i(\tau_i(u)) - u)}\bs{1}_{\{\tau_i(u) \leqslant T_{\omega_i}\}}\right) \\
    &\:\:\:\:\:+ {\bs 1}_{\{i\neq j\}}\frac{q_{ij}}{\omega_i} \int_0^u  \mathbb{P}\left(\bar{Y}_i(T_{\omega_i})\in\mathrm{d}y\right) \int_{0}^\infty \psi_i(\omega_i)\,e^{-\psi_i(\omega_i) z} \int_{u-(y-z)}^\infty \mathbb{P}\left(B_{ij}\in\mathrm{d}\nu\right)\,e^{-\alpha(y-z+\nu-u)}\,{\mathrm d}z\\
    &\:\:\:\:\:+ \sum_{k=1,k\neq i}^d \frac{q_{ik}}{\omega_i} \int_0^u  \mathbb{P}\left(\bar{Y}_i(T_{\omega_i})\in\mathrm{d}y\right) \int_{0}^\infty \psi_i(\omega_i)\,e^{-\psi_i(\omega_i) z}\\
    & \:\:\:\:\:\:\:\:\:\:\cdot\int_{0}^{u-(y-z)} \mathbb{P}\left(B_{ik}\in\mathrm{d}\nu\right)\,\eta_{kj}(u-(y-z+v), \alpha,\beta)\,{\mathrm d}z.
\end{align*}
Next, we multiply the equation by $e^{-\gamma u}$ and integrate over $u\geqslant 0$. 
Recall, from the proof of Lemma \ref{lemma: eta and zeta Levy}, that the LST of the running maximum is given by, for any $\omega > 0$,
\begin{equation*}
    {\mathbb E}\, e^{-\alpha \bar Y_i(T_\omega)}
    = \frac{\psi_i(\omega) - \alpha}{\omega - \varphi_i(\alpha)}
      \frac{\omega}{\psi_i(\omega)},
\end{equation*}
where $\varphi_i(\alpha)$ denotes the Laplace exponent of the L\'evy process corresponding to background state $i\in\defD$. We use Lemma~\ref{lemma: eta and zeta Levy} to analyze the Laplace transform of the first term of the master equation. For the transform of the second and third term, we apply a number of changes of variables, and in addition we interchange the order of integration, so as to obtain
\begin{align}\label{eq: sys eqs nonsub}
    \notag\zeta_{ij}&(\alpha, \beta, \gamma) = {\bs 1}_{\{i = j\}} \frac{1}{\omega_i - \varphi_i(\gamma)} \left( \frac{\varphi_i(\alpha) - \varphi_i(\gamma)}{\gamma - \alpha} - \frac{\varphi_i(\alpha) - \omega_i}{\psi_i(\omega_i) - \alpha} \right) \\
    \notag&\:\:\:\:\:+{\bs 1}_{\{i\neq j\}} \frac{q_{ij}}{\omega_i - \varphi_i(\gamma)} \left(\frac{\mathscr{B}_{ij}(\alpha) - \mathscr{B}_{ij}(\gamma)}{\gamma - \alpha} - \frac{\mathscr{B}_{ij}(\alpha) - \mathscr{B}_{ij}(\psi_i(\omega_i))}{\psi_i(\omega_i) - \alpha}\right)\\
    &\:\:\:\:\:+\sum_{k=1,k\neq i}^d \frac{q_{ik}}{\omega_i - \varphi_i(\gamma)}\left(\mathscr{B}_{ik}(\gamma)\zeta_{kj}(\alpha,\beta,\gamma) - \mathscr{B}_{ik}(\psi_i(\omega_i)) \zeta_{kj}(\alpha,\beta,\psi_i(\omega_i))\right),
\end{align}
for any $\alpha,\gamma \geqslant0$ and $\beta>0$. {A full proof of the derivation of this Laplace transform is given in Appendix \ref{app: proofs lemmas}.

\subsubsection*{Subordinator states}
We now consider states $i\in\defDplus $, for which the right inverse $\psi_i(\cdot)$ is ill-defined and $\bar{Y}_i = Y_i$ almost surely. Hence, the master equation \eqref{eq: master equation eta} simplifies to
\begin{align*}
    \eta_{ij}(u,\alpha,\beta) =\:&  {\bs 1}_{\{i=j\}}\mathbb{E}\left( e^{-\alpha(Y_i(\tau_i(u)) - u)}\bs{1}_{\{\tau_i(u) \leqslant T_{\omega_i}\}}\right) \,+\\
    & {\bs 1}_{\{i\neq j\}}\frac{q_{ij}}{\omega_i} \int_{[0,u]}  \mathbb{P}\left(\bar{Y}_i(T_{\omega_i})\in\mathrm{d}y\right) \int_{(u-y,\infty)} \mathbb{P}\left(B_{ij}\in\mathrm{d}\nu\right)\,e^{-\alpha(y+\nu-u)}\,+\\
    &\sum_{k=1,k\neq i}^d \frac{q_{ik}}{\omega_i} \int_{[0,u]} \mathbb{P}\left(\bar{Y}_i(T_{\omega_i})\in\mathrm{d}y\right) \int_{[0,u-y]} \mathbb{P}\left(B_{ik}\in\mathrm{d}\nu\right)\,\eta_{kj}(u-(y+v), \alpha,\beta).
\end{align*}
We again take the Laplace transform with respect to $u$. {Using Lemma \ref{lemma: eta and zeta Levy subord} from Section \ref{sec: prelim overshoot transform}, the Laplace transform of the first term directly follows.} By interchanging the order of integration in the second and third terms, so that the integral with respect to $u$ becomes the innermost one, and by applying the change of variables $u \to w := u - y + \nu$, we obtain
\begin{align}
    \notag\zeta_{ij}&(\alpha, \beta, \gamma) = {\bs 1}_{\{i = j\}} \frac{1}{\omega_i - \varphi_i(\gamma)} \frac{\varphi_i(\alpha) - \varphi_i(\gamma)}{\gamma - \alpha}\,+ \\
    &\:\:\:\:\:{\bs 1}_{\{i\neq j\}} \frac{q_{ij}}{\omega_i - \varphi_i(\gamma)} \frac{\mathscr{B}_{ij}(\alpha) - \mathscr{B}_{ij}(\gamma)}{\gamma - \alpha} +\sum_{k=1,k\neq i}^d \frac{q_{ik}}{\omega_i - \varphi_i(\gamma)}\,\mathscr{B}_{ik}(\gamma)\,\zeta_{kj}(\alpha,\beta,\gamma), \label{eq: sys eqs sub}
\end{align}
where we used that in this subordinator case we have $\mathbb{E}\,e^{-\alpha \bar{Y}_i(T_{\omega})} = {\omega}/({\omega - \varphi_i(\alpha)})$ for any $\omega > 0$.

\subsubsection*{Solving the system of equations}
The next step is to rewrite the full system, \eqref{eq: sys eqs nonsub} and \eqref{eq: sys eqs sub}, in a more convenient form by moving all linear combinations of $\bs\zeta(\alpha,\beta,\gamma)$ to the left-hand side and the remaining terms to the right. This yields, for suitably defined functions $\kappa_{ij}(\alpha, \beta, \gamma)$,
\begin{equation*}
    (\varphi_i(\gamma) - \omega_i)\,\zeta_{ij}(\alpha,\beta,\gamma) + \sum_{k\neq i} q_{ik} \mathscr{B}_{ik}(\gamma)\,\zeta_{kj}(\alpha, \beta, \gamma) = \kappa_{ij}(\alpha, \beta, \gamma).
\end{equation*}
Introducing the compact notation, for $j\in\defD $,
\[{\bs \zeta}_j(\alpha,\beta,\gamma)\equiv (\zeta_{1j}(\alpha,\beta,\gamma),\ldots,\zeta_{dj}(\alpha,\beta,\gamma))^\top,\:\:\: {\bs \kappa}_j(\alpha,\beta,\gamma)\equiv (\kappa_{1j}(\alpha,\beta,\gamma),\ldots,\kappa_{dj}(\alpha,\beta,\gamma))^\top,\]
the system of equations can be written in matrix-vector form. With the matrix  $\bs F(\gamma)$ as has been defined in \eqref{eq: defF}, we thus obtain, for any $j\in\defD $,
\begin{equation}
    (\bs{F}(\gamma)-\beta\bs{I})\,{\bs \zeta}_j(\alpha,\beta,\gamma) = {\bs \kappa}_j(\alpha,\beta,\gamma).
\end{equation}
Observe that, for given $\alpha\geqslant 0$ and $\beta>0$, the quantities $\eta_{kj}(\alpha,\beta,\psi_i(\omega_i))$ are unknown constants. For that reason, we split $\kappa_{ij}(\alpha,\beta,\gamma)$ into (i)~a term that is fully expressed in terms of the model primitives, and that depends on $\alpha$ and $\gamma$, and (ii)~a term that depends on these unknown quantities and $\alpha$ and $\beta$.
More concretely, we write 
$\kappa_{ij}(\alpha,\beta,\gamma):= \check \kappa_{ij}(\alpha,\gamma)+\bar \kappa_{ij}(\alpha,\beta)$, with
\begin{equation*}
    \check \kappa_{ij}(\alpha,\gamma):= - {\bs 1}_{\{i = j\}} \frac{\varphi_i(\alpha) - \varphi_i(\gamma)}{\gamma - \alpha} -{\bs 1}_{\{i\neq j\}} q_{ij} \frac{\mathscr{B}_{ij}(\alpha) - \mathscr{B}_{ij}(\gamma)}{\gamma - \alpha}
\end{equation*}
for $i\in\defD $, and
\begin{align*}
    \bar \kappa_{ij}(\alpha,\beta) &:= {\bs 1}_{\{i = j\}}\frac{\varphi_i(\alpha) - \omega_i}{\psi_i(\omega_i) - \alpha} + {\bs 1}_{\{i \neq j\}}q_{ij}\frac{\mathscr{B}_{ij}(\alpha) - \mathscr{B}_{ij}(\psi_i(\omega_i))}{\psi_i(\omega_i) - \alpha}\\
    &\:\:\:\:\:+\sum_{k=1,k\neq i}^d q_{ik} \mathscr{B}_{ik}(\psi_i(\omega_i))\, \zeta_{kj}(\alpha,\beta,\psi_i(\omega_i)),
\end{align*}
for $i\in\defDmin $, and $\bar \kappa_{ij}(\alpha,\beta) = 0$ for $i\in\defDplus $.
In evident notation, the following result collects the above findings.
\begin{proposition}\label{prop: system zeta}
    For $\alpha,\gamma\geqslant 0$ and $\beta>0$, 
    \[ (\bs{F}(\gamma)-\beta\bs{I})\,{\bs \zeta}(\alpha,\beta,\gamma) = \bar{\bs \kappa}(\alpha,\beta)+\check{\bs \kappa}(\alpha,\gamma).\]
\end{proposition}
Interestingly, the matrix $ \check {\bs{\kappa}}(\alpha,\gamma)$ satisfies a compact form, as stated in the following lemma.
\begin{lemma}\label{lemma: kappa2}
    For $\alpha,\gamma\geqslant 0$,
    \begin{equation}\label{eq: kappa2eq}
        \check {\bs{\kappa}}(\alpha,\gamma) = \frac{\bs{F}(\gamma) - \bs{F}(\alpha)}{\gamma - \alpha}.
    \end{equation}
\end{lemma}
\begin{proof}
    This is easily verified by breaking down the right-hand side of \eqref{eq: kappa2eq} using \eqref{eq: defF}.
\end{proof}
In order to uniquely characterize ${\bs \zeta}(\alpha,\beta,\gamma)$, we are left with determining, for given $\alpha\geqslant0$ and $\beta>0$, the unknown constants $\bar {\bs{\kappa}}(\alpha,\beta)$. From this point on we fix a background state $j\in\defD $. In the sequel we denote by $\bar {\bs F}(\beta, \gamma)$ the matrix $\bs{F}(\gamma) - \beta\bs{I}$, so that we have the system
\[\bar{\bs F}(\beta,\gamma) \, \bs{\zeta}_j(\alpha,\beta,\gamma)= \bs{\kappa}_j(\alpha,\beta,\gamma).\]

The procedure followed to identify ${\bs \zeta}(\alpha,\beta,\gamma)$ mimics the one
laid out in \cite[\S 3.3]{MB}. We distinguish two steps. 

\medskip 

(i)~We start by applying Cramer's rule for matrices, so as to write, for $i\in\defD $,
\[\zeta_{ij}(\alpha,\beta,\gamma) = \frac{\det \bar {\bs F}_{\bs\kappa_j, i}(\beta, \gamma)}{\det\bar {\bs F}(\beta,\gamma)},\]
where $\bar {\bs F}_{\bs\kappa_j, i}(\beta, \gamma)$ is defined as the matrix $\bar {\bs F}(\beta,\gamma)$ with its $i$-th column replaced by the vector $\bs{\kappa}_j(\alpha, \beta, \gamma)$; note that in the notation $\bar {\bs F}_{\bs\kappa_j, i}(\beta, \gamma)$ we have suppressed the dependence on $\alpha$. 

\medskip 

(ii)~The next step is to analyze the roots of $\det \bar{\bs F}(\beta,\gamma)$ for a fixed $\beta>0$. By \cite[Theorem 1]{IBM}, it follows immediately that, for each $\beta>0$, $\det \bar{\bs F}(\beta,\gamma)$ has exactly $d_-$ roots in the right half of the complex $\gamma$-plane; see also \cite[Lemma 3]{MKD}. We denote these roots by $\gamma_1(\beta),\dots,\gamma_{d_-}(\beta)$ and assume them to be simple. The case of higher multiplicities can be handled using Jordan chains, for which we refer to the detailed treatment in D’Auria et al.\ \cite{IV}.

As $\zeta_{ij}(\alpha, \beta, \gamma)$ is finite for all $\gamma\geqslant0$, any root of the denominator has to be a root of the numerator as well. This implies that we necessarily have
\begin{equation}\label{eq: kappa_j}
     \det\bar {\bs F}_{\bs\kappa_j, i}\left(\beta, \gamma_k(\beta)\right) = 0
\end{equation}
for all $i\in\defD $ and $k\in\defDmin $; recall that we fixed $j\in \defD $. 
Because $\check{\bs \kappa}_j(\alpha,\gamma)$ is a known function of the model primitives, this means that we are left with finding the $d$ entries of the vector $\bar{\bs \kappa}_j(\alpha,\beta)$, of which only the first $d_-$ are unknown.
Observe that, for our given $j\in\defD $, \eqref{eq: kappa_j} seemingly yields $d \cdot d_-$ equations (that are linear in the entries of $\bar{\bs \kappa}_j(\alpha,\beta)$). It can be seen, however, that $d \cdot d_- - d_-$ of these equations are essentially redundant: as argued in \cite[\S III.3]{MB} and \cite{MKD},
the $d \cdot d_-$ equations can be thinned to just $d_-$ equations that `contain all information'. 
This concretely implies that we can simply focus on the subsystem given by the equations 
\begin{equation}\label{eq: sys for kappa1}
    \det \bar {\bs F}_{\bs\kappa_j,1}\left(\beta, \gamma_k(\beta)\right) = 0
\end{equation} 
for $k\in\defDmin $; these are $d_-$ linear equations in $\bar{\kappa}_{ij}(\alpha,\beta)$, $i\in\defDmin $.

\medskip 

We summarize the results of this subsection in the following theorem. 

\begin{theorem}\label{thrm: eta}
    For $\alpha,\gamma\geqslant0$ and $\beta>0$, the double transform $\bs{\zeta}(\alpha,\beta,\gamma)$ is given by
    \begin{equation*}
        \bs{\zeta}(\alpha,\beta,\gamma) = (\bs{F}(\gamma)-\beta\bs{I})^{-1}\left( \bar{\bs \kappa}(\alpha,\beta)+\check{\bs \kappa}(\alpha,\gamma)\right),
    \end{equation*}
    where $\check{\bs{\kappa}}(\alpha,\gamma)$ is given by Lemma $\ref{lemma: kappa2}$ and $\bar{\bs \kappa}(\alpha,\beta)$ satisfies the system of linear equations in \eqref{eq: sys for kappa1}.
\end{theorem}
 A consequence of this theorem is that, for $u\geqslant0$, $\bs{\eta}(u,\alpha,\beta)$ can be numerically obtained by applying an inverse Laplace transformation algorithm to $\bs{\zeta}(\alpha,\beta,\gamma)$ with respect to the parameter $\gamma$.

\begin{remark}\label{remark: eta Ivanovs}{\em 
   An alternative approach to deriving $\bs\eta(u,\alpha,\beta)$ in the absence of subordinators, that is, when $d=d_-$, is presented in \cite[Corollary 4]{IP}. One can verify that this approach and ours yield the same overshoot transform $\bs\eta(u,\alpha,\beta)$ by comparing their respective Laplace transforms with respect to $u$.}\hfill$\Diamond$
\end{remark}

\subsection{Hitting probabilities}\label{sec: hitting probabilities}
The remaining objective is to evaluate the hitting probability matrices ${\bs \delta}_{-}(u_-, u_+, \beta)$ and ${\bs \delta}_{+}(u_-, u_+, \beta)$, which we introduced in the beginning of Section~\ref{sec: decomposition}. For the specific case that none of the L\'evy processes are subordinators the evaluation of these hitting probability matrices is covered by \cite{IP}. In this subsection we show how those results can be extended using the results from Section~\ref{sec: overshoot transform} to cover the more general model which includes subordinators. 

The analysis relies heavily on the (\emph{primary}) \emph{scale matrix}, which extends the primary scale function for spectrally-positive L\'evy processes discussed in Section~\ref{sec: prelim} to the matrix setting. The scale matrix is the matrix-valued function $\bs{W}^{(\beta)}(y): [0,\infty) \to \mathbb{R}^{d\times d_-}$ such that, for any $\alpha\geqslant 0$ and $\beta>0$,
\begin{equation*}
    \int_{0}^\infty e^{-\alpha y} \,\bs{W}^{(\beta)}(y)\,\mathrm{d}y = \left(\bs F(\alpha) - \beta\bs{I}\right)^{-1}
    \begin{pmatrix}
        \bs{I}_- \\
        \bs{0}
    \end{pmatrix},
\end{equation*}
recalling the definition of $\bs F(\alpha)$ from \eqref{eq: defF}, where ${\bs I}_-$ denotes the $d_- \times d_-$ identity matrix and $\bs{0}$ a $d_+ \times d_-$ all-zeroes matrix. We denote the restriction of $\bs{W}^{(\beta)}(y)$ to the first $d_-$ rows by $\bs{W}_-^{(\beta)}(y)$. 
For background on the existence of $\bs{W}^{(\beta)}(y)$ and the invertibility of $\bs{W}_-^{(\beta)}(y)$, see e.g.\ \cite{IP}.

In the remainder of this subsection, we present expressions for ${\bs \delta}_{-}(u_-, u_+, \beta)$ and ${\bs \delta}_{+}(u_-, u_+, \beta)$ in terms of the primary scale matrix. We start with the case ${\bs \delta}_{-}(u_-, u_+, \beta)$, which is simpler since hitting $-u_-$ involves no undershoot in our spectrally-positive setup.

\begin{lemma}\label{lemma: delta_-}
    For $u_-, u_+ \geqslant 0$, with $u_- + u_+ > 0$, and $\beta > 0$,
    \begin{equation}
        {\bs \delta}_{-}(u_-, u_+, \beta) = \begin{pmatrix}\bs{W}^{(\beta)}(u_+)\left(\bs{W}_-^{(\beta)}(u_-+u_+)\right)^{-1} ,& \bs{0}\,\end{pmatrix},
    \end{equation}
    where $\bs{0}$ denotes a $d\times d_+$ zero matrix.
\end{lemma}
\begin{proof}
   The expression for the first $d_-$ rows and columns of $\bs\delta_-(u_-,u_+,\beta)$ follows directly from \cite[Theorem 1]{IP}; this corresponds to the cases in which the background process is in $\defDmin $ both at time 0 and at time $\sigma(u_-)$. Note that \cite{IP} considers a spectrally negative Markov additive process, with only downward jumps, whereas our setting involves upward jumps. Applying these results therefore requires a few elementary transformations. The remaining $d_+$ columns are identically zero, since the process $Y$ cannot drop below the level $-u_-$ while $J$ is in a subordinator background-state. Finally, the entries with $i\in\defDplus, j\in\defDmin$ [was: The rows with index in $\defDplus $] follow from \cite[\S 7.6]{thesis_IV}.
\end{proof}

To derive an expression for ${\bs \delta}_{+}(u_-, u_+, \beta)$, we use the matrix $\bs{P}_+(u,\beta)$ denoting the probabilities of exceeding a level $u\geqslant0$ before killing, whose entries are defined by, for $u\geqslant0, \beta>0$ and $i,j\in\defD $,
\begin{equation}
    P_{+,ij}(u,\beta) := \mathbb{P}_i\left(\tau(u) \leqslant T_{\beta}, J(\tau(u))=j\right).
\end{equation}
Observe that ${\bs P}_{+}(u,\beta)={\bs \eta}(u,0,\beta)$, which is characterized in Theorem~\ref{thrm: eta} via its Laplace transform ${\bs \zeta}(0,\beta,\gamma)$.
In the following lemma we express ${\bs \delta}_{+}(u_-, u_+, \beta)$ in terms of ${\bs P}_{+}(u,\beta)$ and ${\bs \delta}_{-}(u_-, u_+, \beta)$. 

\begin{lemma}\label{lemma: delta_+ in terms of delta_-}
    For any $u_-, u_{+} \geqslant 0$ with $u_-+u_+>0$ and $\beta>0$
    \begin{equation}\label{eq: delta_+ in terms of delta_-}
        {\bs \delta}_{+}(u_-, u_+, \beta) = {\bs P}_+(u_+, \beta) - {\bs \delta}_{-}(u_-,u_+,\beta)\,{\bs P}_{+}(u_-+u_+, \beta)
    \end{equation}
\end{lemma}
\begin{proof}
    We start by noting that we can alternatively write $ \delta_{-,ij}(u_-, u_+, \beta)$ and $ \delta_{+,ij}(u_-, u_+, \beta)$ as
\begin{align*}
    \delta_{-,ij}(u_-, u_+, \beta) &= \mathbb{E}_i\left(e^{-\beta \sigma(u_-)}\bs{1}_{\{\sigma(u_-) \leqslant \tau(u_+),\, J(\sigma(u_-))=j\}}\right) \\
    \delta_{+,ij}(u_-, u_+, \beta) &= \mathbb{E}_i\left(e^{-\beta \tau(u_+)}\bs{1}_{\{\tau(u_+) \leqslant \sigma(u_-),\, J(\tau(u_+))=j\}}\right);
\end{align*}
cf.\ \cite[Lemma 5.2]{MB}.
As a consequence, by some elementary rearranging, we observe that the entries of $\bs\delta_{+}(u_-, u_+, \beta)$ obey
\begin{align*}
    \delta_{+,ij}(u_-, u_+, \beta) &= \mathbb{E}_i\left(e^{-\beta \tau(u_+)}\bs{1}_{\{\tau(u_+) <\infty,\, J(\tau(u_+))=j\}}\right) - \mathbb{E}_i\left(e^{-\beta \tau(u_+)}\bs{1}_{\{\sigma(u_-)<\tau(u_+) ,\, J(\tau(u_+))=j\}}\right),
\end{align*}
where the first term on the right-hand side is simply $P_{+,ij}(u_+, \beta)$.
We continue by analyzing the second term on the right side. Considering the individual entries of the matrix, we obtain, as a consequence of the strong Markov property,
\begin{align*}
    \mathbb{E}_i&\left[e^{-\beta \tau(u_+)}\bs{1}_{\{\sigma(u_-)<\tau(u_+) ,\, J(\tau(u_+)) = j\}}\right]= \mathbb{P}_i(\sigma(u_-) < \tau(u_+) \leqslant T_{\beta}, J(\tau(u_+)) = j) \\
    &= \sum_{k=1}^d \mathbb{P}_i(\sigma(u_-)\leqslant \min\{\tau(u_+), T_{\beta}\},\,J(\sigma(u_-)) = k)\, \mathbb{P}_k(\tau(u_- + u_+)\leqslant T_{\beta}, J(\tau(u_- + u_+)) = j).
\end{align*}
In matrix-vector form this identity can be compactly written as 
\begin{equation*}
    \mathbb{E}\left(e^{-\beta \tau(u_+)}\bs{1}_{\{\tau(u_+) > \sigma(u_-),\, J(\tau(u_+))\}}\right) = {\bs \delta}_{-}(u_-,u_+,\beta)\,{\bs P}_{+}(u_-+u_+, \beta).
\end{equation*}
Upon combining the above we find the claimed statement. 
\end{proof}
\begin{remark}\label{remark: hitting probs non sub}{\em
    In the absence of subordinators, i.e., $\defDmin=\defD$, the evaluation of the hitting probability matrices for spectrally-negative MAPs was already covered by \cite{IP}. In this case the primary scale matrix ${\bs W}^{(\beta)}(y)$ is a square $d\times d$ matrix and ${\bs \delta}_-(u_-,u_+,\beta)$ does not have any zero columns as the level $-u_-$ can be crossed in any state. In addition the \textit{secondary scale matrix} is often considered as well, defined by
    \begin{equation*}
        \bs{Z}^{(\beta)}(u) := \bs I - \int_{0}^u \bs{W}^{(\beta)}(y)\,\mathrm{d}y\left(Q - \beta\bs{I}\right).
    \end{equation*}
    It can be verified, by tedious calculations, that in this case ${\bs P}_+(u, \beta)$ can be expressed in terms of both scale matrices as 
    \begin{equation}\label{eq: P_+}
        \bs{P}_{+}(u,\beta) = \bs{Z}^{(\beta)}(u) + \bs{W}^{(\beta)}(u)\,\bar{\bs\kappa}(0,\beta)
    \end{equation}
    by taking the Laplace transforms (with respect to $u$) of both sides and show that they are the same using Theorem \ref{thrm: eta} for $\alpha=0$.
    
    Substituting \eqref{eq: P_+} into \eqref{eq: delta_+ in terms of delta_-} then yields for ${\bs \delta}_+(u_-, u_+, \beta)$
    \begin{align*}
        {\bs \delta}_{+}(u_-, u_+, \beta) &=  \bs{Z}^{(\beta)}(u_+) + \bs{W}^{(\beta)}(u_+)\,\bar{\bs\kappa}(0,\beta)\:
        -\\&\hspace{1.3cm}\bs{W}^{(\beta)}(u_+)\left(\bs{W}^{(\beta)}(u_-+u_+)\right)^{-1}\left(\bs{Z}^{(\beta)}(u_-+u_+) + \bs{W}^{(\beta)}(u_-+u_+)\,\bar{\bs\kappa}(0,\beta)\right)
        \\
        &= \bs{Z}^{(\beta)}(u_+) - \bs{W}^{(\beta)}(u_+)\left(\bs{W}^{(\beta)}(u_- + u_+)\right)^{-1}\bs{Z}^{(\beta)}(u_- + u_+).
    \end{align*}
    This result is consistent with \cite[Corollary~3]{IP}, after incorporating killing and setting $\alpha=0$ in the formulation given there. Interestingly, the above derivation reveals that in order to derive an expression for $\bs{\delta}_+(u_-, u_+, \beta)$ we actually do not need to know the precise form of the matrix $\bar{\bs\kappa}(0,\beta)$ in  \eqref{eq: P_+}, in that the only crucial property is that we have  \eqref{eq: P_+} for some matrix $\bar{\bs\kappa}(0,\beta)$ which does not depend on the level $u$. \hfill$\Diamond$
    }
\end{remark}

\section{Idle time and lost work}\label{sec: extensions}
When $Y$ is a Markov-modulated compound Poisson process with strictly negative drifts, the workload experiences idle periods during which the system is empty, and any arriving work that exceeds the available capacity $K$ is lost. In this section, we analyze the idle time and the amount of lost work in a finite-capacity queue driven by Markov-modulated compound Poisson input. Throughout, $Y_i$ denotes a compound Poisson process with arrival rate $\lambda_i$, non-negative jump sizes distributed as a generic random variable $B_i$, and an added negative drift $r_i$.

Let $I\equiv (I(t))_{t\geqslant0}$ denote the total idle time up to $t$, and $L\equiv (L(t))_{t\geqslant0}$ the cumulative lost work due to overshoots; both processes are inherently non-negative and non-decreasing.  
In this section, we study the joint distribution of $I(T_\beta)$ and $L(T_\beta)$ along with the quantities considered previously. Specifically, for all $i,j\in\defD $ and $x\in[0,K]$, our object of interest is the LST
\begin{equation}
    \tilde\chi_{ij}(x) \equiv \tilde\chi_{i,j}(x,\bs\alpha,\beta) := \e_{x,i}\left( e^{-\alpha_1 V(T_\beta) - \alpha_2 I(T_\beta) - \alpha_3 L(T_\beta)} \bs{1}_{\{J(T_\beta)=j\}} \right),
\end{equation}
for $\bs\alpha=(\alpha_1,\alpha_2,\alpha_3)$, $\alpha_k\geqslant0$ for $k=1,2,3$, and $\beta > 0;$ the subscripts in the rightmost expression indicate, as before, that the initial condition is $V(0)=x$ and $J(0)=i$. 
In our analysis of this object we make use of the following overshoot LST: for any $u\in[0,K]$,
\begin{equation}
    \tilde{\eta}_{ij}(u) \equiv \tilde{\eta}_{ij}(u, \alpha, \beta) := {\mathbb E}_{i}\left(e^{-\alpha \left(Y(\tau(u)) - u\right)}\,\bs{1}_{\{\tau(u) \leqslant \min\{\sigma(K-u), T_{\beta}\}, J(\tau(u)) = j\}}\right).
\end{equation}
It is noted that $\tilde{\bs \eta}(u,0,\beta) = {\bs \delta}_+(K-u,u,\beta)$.  
In addition, $\tilde{\bs \eta}(u)$ is closely related to the overshoot LST $\bs{\eta}(x)$ defined in \eqref{eq: eta definition} and characterized in Section~\ref{sec: overshoot transform}. In particular, the following lemma shows that $\tilde{\bs \eta}(u)$ can be expressed in terms of $\bs{\eta}(\cdot)$, similar to what was done in Lemma~\ref{lemma: delta_+ in terms of delta_-} for ${\bs \delta}_{+}(u_-, u_+, \beta)$ in terms of ${\bs P}_+(\cdot, \beta)$.

\begin{lemma}\label{lemma: eta tilde}
    For $u\in[0,K], \alpha \geqslant 0$ and $\beta>0$,
    \begin{equation}\label{eq: tilde eta}
        \tilde{\bs \eta}(u) = \bs{\eta}(u) - \bs{\delta}_{-}(K-u)\,\bs{\eta}(K).
    \end{equation}
\end{lemma}
\begin{proof}
    As a  first step, observe that we can decompose $\tilde{\eta}_{ij}(u)$ as
    \begin{align*}
        \tilde{\eta}_{ij}(u) =\eta_{ij}(u) - {\mathbb E}_{i}\left(e^{-\alpha \left(Y(\tau(u)) - u\right)}\,\bs{1}_{\{\sigma(K-u) < \tau(u) \leqslant T_{\beta}, J(\tau(u)) = j\}}\right).
    \end{align*}
    Applying the strong Markov property to the second term,
    \begin{align*}
        {\mathbb E}_{i}&\left(e^{-\alpha \left(Y(\tau(u)) - u\right)}\,\bs{1}_{\{\sigma(K-u) < \tau(u) \leqslant T_{\beta}, J(\tau(u)) = j\}}\right) \\
        &= \sum_{k=1}^d\delta_{-,ik}(K-u)\,{\mathbb E}_{k}\left(e^{-\alpha \left(Y(\tau(K)) - K\right)}\,\bs{1}_{\{\tau(K) \leqslant T_{\beta}, J(\tau(K)) = j\}}\right) = \sum_{k=1}^d\delta_{-,ik}(K-u)\,\eta_{kj}(K).
    \end{align*}
    The claimed statement \eqref{eq: tilde eta} follows immediately.
\end{proof}

In the following lemma, which can be considered to be a counterpart of Lemma \ref{lemma: decomp chi(x)}, we show how $\tilde{\bs \chi}(x)$ can be decomposed in terms of $\tilde{\bs \chi}(0)$ and $\tilde{\bs \chi}(K)$. Note that at $x=0$ (i.e., the workload process' lower boundary) this decomposition trivially holds, as ${\bs \delta_-(0)}$ is simply ${\bs I}$ and the entries of $\tilde{\bs \eta}(K, \alpha_3, \beta)$ and $\bs\delta_{\star}(0)$ are all zero. 

\begin{lemma}\label{lemma: decomp chi(x) idle and lost}
    For $\alpha \geqslant0, \beta>0$ and $x\in
    [0,K]$,
    \begin{equation}\label{eq: chi(x) idle and lost}
        \tilde{\bs \chi}(x) = \bs\delta_{-}(x)\,\tilde{\bs \chi}(0) + \tilde{\bs \eta}(K-x, \alpha_3, \beta)\,\tilde{\bs \chi}(K) + \bs\delta_{\star}(x).
    \end{equation}
\end{lemma}
\begin{proof}
By the strong Markov property we have
\begin{equation*}\label{eq: chi_ij(x) idle and lost}
\tilde\chi_{ij}(x) = \sum_{k=1}^d \delta_{-,ik}(x)\,\tilde\chi_{kj}(0) + \sum_{k=1}^d \tilde{\eta}_{ik}(K-x)\,\tilde\chi_{kj}(K) + \delta_{\star ,ij}(x).
\end{equation*}
The validity of this identity can be justified as follows.
The first term on the right-hand side corresponds to the case in which the first event is hitting the level $0$. The second term corresponds to the case in which the first event is hitting the level $K$. In this case, $L(\cdot)$ increases by the amount of the overshoot at time $\tau(K-x)$. Finally, $\delta_{\star ,ij}(x)$ denotes the case in which killing occurs before the process hits either $0$ or $K$. In the first and the last case, both $I(\cdot)$ and $L(\cdot)$ do not increase up to $\sigma(x)$ or $T_{\beta}$, respectively. Rewriting this decomposition in matrix notation yields $\eqref{eq: chi(x) idle and lost}$.
\end{proof}

Finally, we determine the $2d^2$ entries in $\tilde{\bs \chi}(0)$ and $\tilde{\bs \chi}(K)$. To this end, we decompose $\tilde{\bs \chi}(0)$, with $\tilde{\bs \omega} := \beta\bs{1} + \bs{q} + \bs{\lambda}$. 
In the setting considered, the driving process $Y$ being a Markov-modulated compound Poisson process, we can distinguish between three types of events: the arrival of a job,  a background-state transition, and killing. We thus obtain
\begin{align*}
    \tilde\chi_{ij}(0) &= \frac{\tilde\omega_i}{\tilde\omega_i+\alpha_2} \Bigg(\frac{\beta}{\tilde{\omega}_i}+\frac{\lambda_i}{\tilde{\omega}_i} \left(\int_{[0,K]}\mathbb{P}(B_i\in\mathrm{d}y)\tilde\chi_{ij}(y) + \int_{(K,\infty)}\mathbb{P}(B_i\in\mathrm{d}y)e^{-\alpha_3(y-K)}\tilde\chi_{ij}(K)\right)\,+\,\\
    &\:\:\:\:\:\:\quad\sum_{k\neq i}\frac{q_{ik}}{\tilde{\omega}_i} \left(\int_{[0,K]}\mathbb{P}(B_{ik}\in\mathrm{d}y)\tilde\chi_{kj}(y) + \int_{(K,\infty)}\mathbb{P}(B_{ik}\in\mathrm{d}y)e^{-\alpha_3(y-K)}\tilde\chi_{kj}(K)\right)\Bigg);
\end{align*}
the factor ${\tilde\omega_i}/({\tilde\omega_i+\alpha_2})$ here comes from $\mathbb{E}(e^{-\alpha_2 I(T_{\tilde\omega_i})})$ and represents the increase of the idle time until the first event.
Substituting \eqref{eq: chi(x) idle and lost} for $\tilde\chi_{ij}(y)$,
\begin{align}
    \tilde\chi_{ij}(0) &= \frac{\tilde\omega_i}{\tilde\omega_i+\alpha_2} \Bigg(\frac{\beta}{\tilde{\omega}_i}+\frac{\lambda_i}{\tilde{\omega}_i} \int_{[0,K]}\mathbb{P}(B_i\in\mathrm{d}y)\left(\sum_{\ell=1}^d\delta_{-, i\ell}(y)\tilde\chi_{\ell j}(0) + \sum_{\ell=1}^d\delta_{+, i\ell}(y)\tilde\chi_{\ell j}(K) + \delta_{\star,ij}(y)\right)\,+\,\notag\\ 
    &\quad\:\:\:\:\:\:\frac{\lambda_i}{\tilde{\omega}_i}  \int_{(K,\infty)}\mathbb{P}(B_i\in\mathrm{d}y)e^{-\alpha_3(y-K)}\tilde\chi_{ij}(K)\,+\,\notag\\
    &\quad\:\:\:\:\:\:\sum_{k\neq i}\frac{q_{ik}}{\tilde{\omega}_i} \int_{[0,K]}\mathbb{P}(B_{ik}\in\mathrm{d}y)\left(\sum_{\ell=1}^d\delta_{-, k\ell}(y)\tilde\chi_{\ell j}(0) + \sum_{\ell=1}^d\delta_{+, k\ell}(y)\tilde\chi_{\ell j}(K) + \delta_{\star,kj}(y)\right) \,+\,\notag\\
    &\quad\:\:\:\:\:\:\sum_{k\neq i}\frac{q_{ik}}{\tilde{\omega}_i} \int_{(K,\infty)}\mathbb{P}(B_{ik}\in\mathrm{d}y)e^{-\alpha_3(y-K)}\tilde\chi_{kj}(K)\Bigg).\notag
\end{align}
Using Lemma \ref{lemma: decomp chi(x) idle and lost} with $x=K$, we can now rewrite $\tilde{\bs \chi}(K)$, thus yielding a system of linear equations in $\tilde\chi_{ij}(0)$ and $\tilde\chi_{ij}(K)$. This means that we have identified matrices $\tilde{\bs P}$, ${{\smash{\tilde{\bs b}}}^{\{0\}}}$ and ${{\smash{\tilde{\bs b}}}^{\{K\}}}$ so that
\begin{equation*}
    \left(\begin{array}{c} \tilde{\bs \chi}(0) \\ \tilde{\bs \chi}(K)\end{array} \right)= \tilde{\bs P}\left(\begin{array}{c} \tilde{\bs \chi}(0) \\ \tilde{\bs \chi}(K)\end{array} \right) + 
    \left(\begin{array}{l} {{\smash{\tilde{\bs b}}}^{\{0\}}} \\ {{\smash{\tilde{\bs b}}}^{\{K\}}}\end{array} \right).
\end{equation*}
Now it is readily verified that $\bs{I} - \tilde{\bs P}$ is diagonally dominant, essentially following the proof of Lemma~\ref{lemma: diagdom}, which implies that it is invertible. We obtain the following result, where we note that $\bs\delta_{-}(x)$ and $\bs\delta_{\star}(x)$ were already determined in Section \ref{sec: aux}.

\begin{proposition}\label{prop: decomp chi(0) CPP}
       The matrix $\tilde{\bs\chi}(x)$ is given by \eqref{eq: chi(x) idle and lost}, where $\tilde{\bs \eta}(K-x, \alpha_3, \beta)$ follows from Lemma $\ref{lemma: eta tilde}$, and 
\begin{align*}
    \left(\begin{array}{c} \tilde{\bs \chi}(0) \\ \tilde{\bs \chi}(K)\end{array} \right) = \left(\,\bs{I} - \tilde{\bs P}\,\right)^{-1} \left(\begin{array}{l} {{\smash{\tilde{\bs b}}}^{\{0\}}} \\ {{\smash{\tilde{\bs b}}}^{\{K\}}}\end{array} \right).
\end{align*}
\end{proposition}

\section{Numerical examples}\label{sec: num}
In this section, we present numerical results for three representative examples, reporting the mean and variance of the workload, as well as the probabilities of an empty or full buffer over time. These quantities are obtained by first determining their Laplace transforms evaluated at the exponentially distributed time $T_\beta$, and subsequently computed at selected {\it deterministic} time points using numerical Laplace inversion~\cite[\S 8]{AW2}. Implementation details are provided in Appendix~\ref{app: notes on implementation}.

Throughout this section we consider instances with $d=2$ background states and in which the workload process does not increase at transitions of the background process, i.e., $B_{ij} \equiv 0$ for all $i,j\in\{1,2\}$. 
In the first two instances $J$ is recurrent with unitary transition rates, i.e., $q_{12}=q_{21}=1$, whereas in the final instance the second state is absorbing.
Furthermore, all L\'evy processes $Y_i$ considered are either a compound Poisson process (with or without drift), a Brownian motion or a combination of the two. For each instance we  present the Laplace exponents of the L\'evy processes, which are necessarily of the form
\begin{equation}\notag
    \varphi_i(\alpha) = -r_i\alpha + \tfrac{1}{2}\varsigma_i^2\alpha^2 -\lambda_i\left(1-\mathscr{B}_i(\alpha)\right),
\end{equation}
for $\alpha\geqslant0$ and $i\in\{1,2\}$. In other words, when $J$ is  in state $i$, $r_i\in\mathbb{R}$ is the (deterministic) drift, $\varsigma_i\geqslant0$ the `Brownian coefficient', $\lambda_i\geqslant0$ the job arrival rate, and $\mathscr{B}_i(\alpha)$ the LST of the job-size distribution. The instances considered in this section are chosen specifically because they admit exact and concise expressions for the scale matrices ${\boldsymbol{W}}^{(\beta)}$. 

For each instance, we present a figure displaying for $J(0)\in\{1,2\}$ the mean workload and its variance as functions of time, as well as the probability of an empty or full buffer where these are non-trivial. Each plot also includes Monte Carlo-based estimates of the corresponding metrics for comparison: the marks at the integer time points represent the values obtained from numerical Laplace inversion, while the continuous lines show the corresponding Monte Carlo–based estimates. Throughout, we fix the capacity parameter at $K=4$, while allowing the initial workload $x$ to vary.

\subsubsection*{Instance 1}
We first consider a Markov-modulated compound Poisson process; the job sizes in state $i$ are exponentially distributed with mean $\mu_i^{-1}$. 
The parameters chosen are ${\bs r}=(-1,0)$, ${\bs \lambda}=(1,2)$, ${\bs\mu}=(1,1)$, and (evidently) ${\bs\varsigma} =(0,0)$.  We therefore have
\begin{equation*}
   \left(
        \varphi_1(\alpha),
        \varphi_2(\alpha)
    \right) 
    =
    \left(
        \alpha -\frac{\alpha}{1+\alpha},
        -2\frac{\alpha}{1+\alpha}
    \right).
\end{equation*}
In Figure \ref{fig: Model 6} we have plotted the results. The main conclusion is that our numerical results closely match their simulation-based counterparts.

\begin{figure}[ht]
    \centering
    \begin{subfigure}{0.49\textwidth}
        \includegraphics[width=\textwidth]{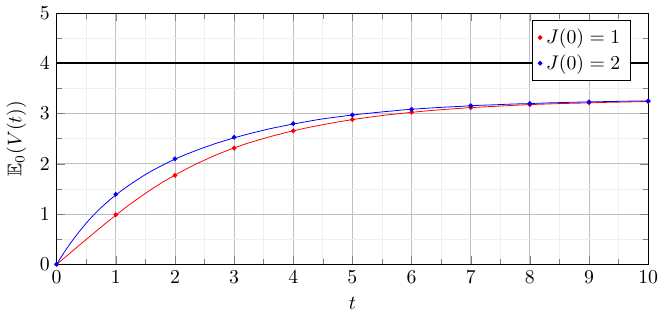}
        \caption{\label{fig: Model 6 EXP}}
    \end{subfigure}
      \hfill
    \begin{subfigure}{0.49\textwidth}
        \includegraphics[width=\textwidth]{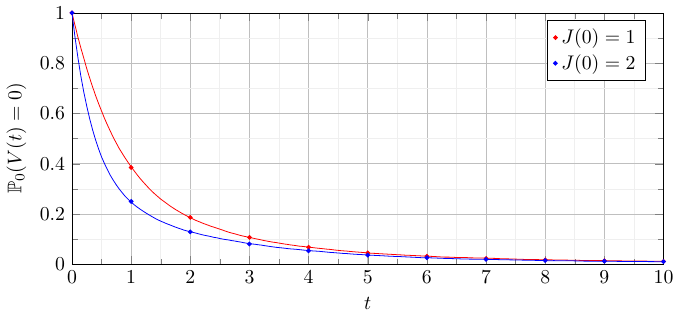}
        \caption{\label{fig: Model 6 empty}}
    \end{subfigure}
    \hfill
    \begin{subfigure}{0.49\textwidth}
        \includegraphics[width=\textwidth]{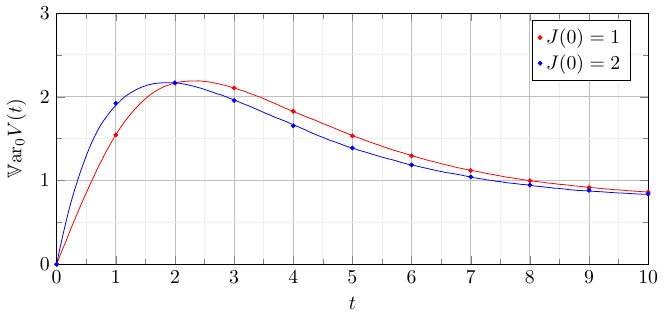}
        \caption{\label{fig: Model 6 var}}
    \end{subfigure}
   \hfill
    \begin{subfigure}{0.49\textwidth}
        \includegraphics[width=\textwidth]{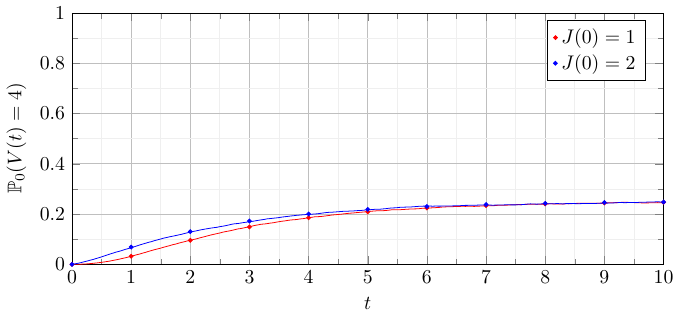}
        \caption{\label{fig: Model 6 full}}
    \end{subfigure}
    \caption{\label{fig: Model 6} Results for instance 1. Left panels: mean ({\sc a}) and variance ({\sc c}) of the workload; right panels: probability of an empty system ({\sc b}) and a full system ({\sc d}). All quantities are shown as functions of time, with  initial workload $x=0$, for different initial background states. }
\end{figure}

\subsubsection*{Instance 2}
For the second instance we consider what could be termed an on/off Brownian motion with negative unitary drift. Concretely, we have $r_1=-1, \varsigma_1=1$ and $\lambda_1=0$, and $Y_2 \equiv 0$, so that
\begin{equation*}
   \left(
        \varphi_1(\alpha) ,
        \varphi_2(\alpha)
    \right)
    =
    \big(
        \alpha + \tfrac{1}{2}\alpha^2,
        0
    \big).
\end{equation*}
Figure \ref{fig: Model 4} shows the mean workload and its variance. For comparison, the purple curves depict the mean and variance of a workload system driven by a single Brownian motion with negative unit drift, obtained by integrating \eqref{eq: dens W init x Levy} and using numerical Laplace inversion. The scale functions in \eqref{eq: dens W init x Levy} are obtained by applying a partial fraction decomposition to $(\varphi(\alpha) - \beta)^{-1}$ and then invoking standard inverse Laplace transform techniques. The dashed lines in Figure~$\ref{fig: Model 4}$ depict the stationary mean and variance of the workload for the single Brownian motion case, as given in \cite[Proposition~11.1]{DM}.
Once again, we observe that our numerical results are in close agreement with the simulation-based estimates. Furthermore, we see that the mean workload and variance of the instance converge to the same steady-state values as the single Brownian motion, albeit more slowly, as expected.

\begin{figure}[ht]
    \centering
    \begin{subfigure}{0.49\textwidth}
        \includegraphics[width=\textwidth]{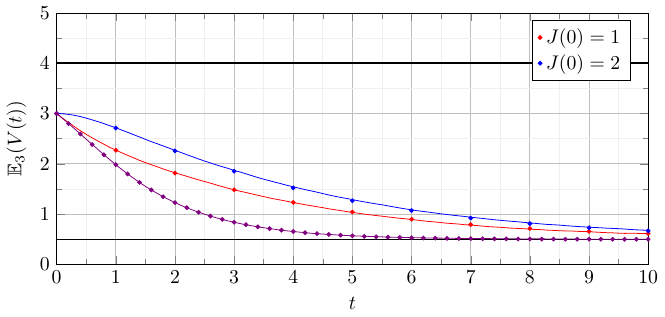}
        \caption{\label{fig: Model 4 EXP}}
    \end{subfigure}
    \hfill
    \begin{subfigure}{0.49\textwidth}
        \includegraphics[width=\textwidth]{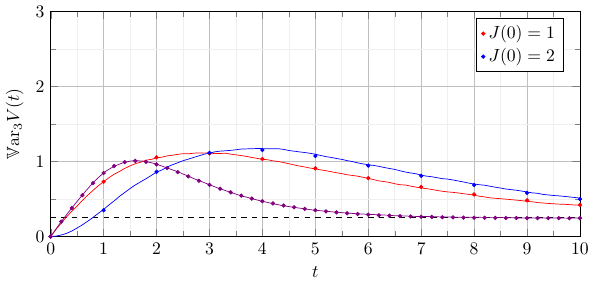}
        \caption{\label{fig: Model 4 var}}
    \end{subfigure}
    \caption{\label{fig: Model 4} Results for instance 2. Mean ({\sc a}) and variance ({\sc b}) of the workload, as functions of time, with initial workload $x=3$, for different initial background states.   The purple line corresponds to $d=1$ and $\varphi(\alpha)=\alpha + \frac{1}{2}\alpha^2$; the dashed lines are the mean and variance of the corresponding workload in stationarity.}
\end{figure}

\subsubsection*{Instance 3}
The third and final instance is a combination of a Brownian motion (state~1) and a compound Poisson process with negative drift and exponentially distributed jumps (state~2), in which state~2 is absorbing, i.e., $q_{12} = 1$ but $q_{21} = 0$. In terms of the model parameters we therefore have $\bs{r} = (-1, -1)$, $\bs{\lambda} = (0,1)$, $\bs{\mu} = (-,1)$ and $\bs{\varsigma} = (1,0)$. We therefore have
\begin{equation*}
    \left(
        \varphi_1(\alpha) ,
        \varphi_2(\alpha)
    \right)
    =
    \left( \alpha+\tfrac{1}{2}\alpha^2, \alpha-\frac{\alpha}{1+\alpha}\right).
\end{equation*}
Figure \ref{fig: Model 7} again confirms that our numerical procedure provides highly accurate output. 

\begin{figure}[ht]
    \centering
    \begin{subfigure}{0.49\textwidth}
        \includegraphics[width=\textwidth]{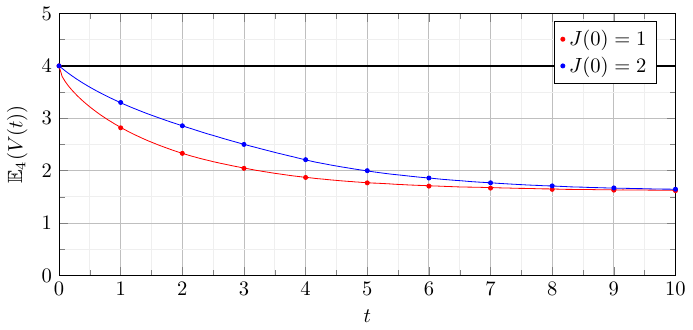}
        \caption{\label{fig: Model 7 EXP}}
    \end{subfigure}
    \hfill
    \begin{subfigure}{0.49\textwidth}
        \includegraphics[width=\textwidth]{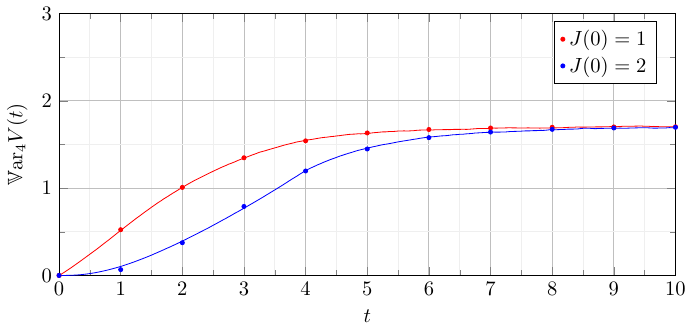}
        \caption{\label{fig: Model 7 var}}
    \end{subfigure}
    \hfill
    \begin{subfigure}{0.49\textwidth}
        \includegraphics[width=\textwidth]{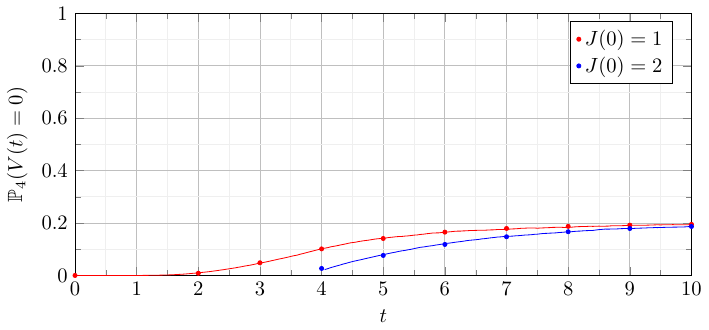}
        \caption{\label{fig: Model 7 empty}}
    \end{subfigure}
    \caption{\label{fig: Model 7} Results for instance 3. Mean ({\sc a}) and variance ({\sc b}) of the workload, and probability of an empty system ({\sc c}), as functions of time, with initial workload $x=K=4$, for different initial background  states. }
\end{figure}

\section{Discussion and concluding remarks}
In this paper, we have developed a procedure to evaluate the LST of the workload for a doubly reflected spectrally one-sided MAP at an exponentially distributed time. The framework we consider is general in several important respects. In particular, no assumptions are imposed on the structure of the background process, which need not be irreducible. Moreover, the spectrally one-sided L\'evy processes underlying the MAP are allowed to be subordinators. Finally, apart from the initial background state, also the initial workload may be chosen arbitrarily, so that the analysis does not require the system to start in an empty or full state.

Arguably, the most significant restriction of the setup considered is the spectrally one-sided nature of the driving MAP. A natural next step would be to address the case in which jumps in one direction have a general distribution, while jumps in the other direction are of phase-type or are assumed to have rational Laplace transforms. For the case of a singly-reflected L\'evy process, this approach has been successfully implemented~\cite{LM}. We also refer, for example, to \cite{ASG,PAR} for illustrations of highly general classes of spectrally two-sided L\'evy processes that may permit relatively explicit evaluation.

In the case where $Y$ is a Markov-modulated compound Poisson process, we have also succeeded in analyzing the cumulative amounts of lost work and idle time (at killing). Another open problem remains the identification of the joint LST that, for $Y$ being a general spectrally one-sided MAP, additionally includes the regulators $U^-(T_\beta)$ and $U^+(T_\beta)$, as defined in \eqref{eq: regu}; cf.\ the concept of {\it loss rate} in \cite{AAGP}.

\appendix

\section{}
\subsection{Proofs}\label{app: proofs lemmas} In this appendix, we collect the proofs of several lemmas and equations.

\subsubsection*{Proof of Lemma \ref{lemma: LST W Levy inti x}}\label{app: proof LST Levy}
Let $\alpha\geqslant0$. By multiplying \eqref{eq: dens W init x Levy} by $e^{-\alpha y}$, integrating $y$ over $(0,K)$ and adding \eqref{eq: empty prob W init x Levy},
\begin{align*}
    \mathbb{E}_x&\left(e^{-\alpha V(T_\beta)}\right) = \mathbb{P}_x\left(V(T_\beta) = 0\right) + \int_{(0,K)}e^{-\alpha y}\mathbb{P}_x\left(V(T_\beta)\in\mathrm{d}y\right) \\
    &= Z^{(\beta)}(K-x)\frac{W^{(\beta)}(0)}{W^{(\beta)}(K)} + \frac{Z^{(\beta)}(K-x)}{W^{(\beta)}(K)}\int_0^K e^{-\alpha y}\frac{\mathrm{d}}{\mathrm{d}y}W^{(\beta)}(y)\,\mathrm{d}y -\beta\int_x^K e^{-\alpha y}W^{(\beta)}(y-x)\,\mathrm{d}y.
\end{align*}
Applying integration by parts, this is equal to
\begin{align*}
    Z^{(\beta)}&(K-x)\frac{W^{(\beta)}(0)}{W^{(\beta)}(K)} + \frac{Z^{(\beta)}(K-x)}{W^{(\beta)}(K)}\left(\left[e^{-\alpha y}W^{(\beta)}(y)\right]_0^K+\int_0^K \alpha e^{-\alpha y}W^{(\beta)}(y)\,\mathrm{d}y\right) \\
    &\:\:\:\:\:-\beta\int_x^K e^{-\alpha y}W^{(\beta)}(y-x)\,\mathrm{d}y\\
    &= e^{-\alpha K}Z^{(\beta)}(K-x) + \frac{Z^{(\beta)}(K-x)}{W^{(\beta)}(K)}\int_0^K \alpha e^{-\alpha y}W^{(\beta)}(y)\,\mathrm{d}y -\beta e^{-\alpha x}\int_0^{K-x} e^{-\alpha y}W^{(\beta)}(y)\,\mathrm{d}y,
\end{align*}
which is equal to the LST in \eqref{eq: LST Levy W init x}. \hfill$\square$

\subsubsection*{Proof of Lemma \ref{lemma: decomp chi(0) and chi(K)}}\label{app: proof decomp chi(0) and chi(K)}
Let $T_{(1)}$ denote the first time the background process $J$ transitions to a different state. 
Recall that $V_i$ corresponds to the L\'evy process $Y_i$ reflected at the boundaries $0$ and $K$. Clearly, conditional on (i)~$J(0)=i$, $J(T_{(1)}) = k$ (i.e., the first transition is from $i$ to $k$), and (ii)~$T_{(1)} < T_\beta$ (i.e., the first transition happens before killing), $V(T_{(1)})$ is distributed as $\hat V_{ik}:=\min\{V_i(T_{\omega_i}) + B_{ik}, K\}$.

By conditioning on the first event, which is either a transition of the background process  or killing, and the strong Markov property, we can decompose $\chi_{ij}(0)$ into
\begin{align*}
    \chi_{ij}(0) &= \bs{1}_{\{i=j\}}\frac{\beta}{\omega_i}\,\mathbb{E}_0\left(e^{-\alpha V_i(T_{\omega_i})}\right)+ \sum_{k=1, k\neq i}^d\frac{q_{ik}}{\omega_i} \int_{[0,K]}\mathbb{P}_{0,i}\left(V(T_{(1)})\in \mathrm{d}y\,|\,J(T_{(1)})=k, T_{(1)}<T_\beta\right)\,\chi_{kj}(y) \\
    &= \bs{1}_{\{i=j\}}\frac{\beta}{\omega_i}\,\mathbb{E}_0\left(e^{-\alpha V_i(T_{\omega_i})}\right)+ \sum_{k=1, k\neq i}^d\frac{q_{ik}}{\omega_i} \int_{[0,K]}\mathbb{P}_{0}\left(\hat V_{ik}\in \mathrm{d}y\right)\,\chi_{kj}(y).
\end{align*}
Substituting the decomposition for $\chi_{kj}(y)$ from Lemma \ref{lemma: decomp chi(x)}, we obtain
\begin{align*}
    \chi_{ij}(0) &= \bs{1}_{\{i=j\}}\frac{\beta}{\omega_i}\,\mathbb{E}_0\left(e^{-\alpha V_i(T_{\omega_i})}\right)+ \sum_{k=1, k\neq i}^d \sum_{\ell=1}^d\frac{q_{ik}}{\omega_i} \int_{[0,K]}\mathbb{P}_{0}\left(\hat V_{ik}\in \mathrm{d}y\right)\delta_{-,k\ell}(y)\chi_{\ell j}(0)\:+ \\
    & \sum_{k=1, k\neq i}^d \sum_{\ell=1}^d\frac{q_{ik}}{\omega_i} \int_{[0,K]}\mathbb{P}_{0}\left(\hat V_{ik}\in \mathrm{d}y\right)\delta_{+,k\ell}(y)\chi_{\ell j}(K)   + \sum_{k=1, k\neq i}^d \frac{q_{ik}}{\omega_i}\int_{[0,K]}\mathbb{P}_{0}\left(\hat V_{ik}\in \mathrm{d}y\right)\delta_{\star,kj}(y).
\end{align*}
From the definitions of ${\bs P}^{\{0,0\}}, {\bs P}^{\{0,K\}}$ and ${\bs b}^{\{0\}}$ we conclude ${\bs \chi}(0) = {\bs P}^{\{0,0\}}{\bs \chi}(0) + {\bs P}^{\{0,K\}}{\bs \chi}(K) + {\bs b}^{\{0\}}.$ The corresponding equation for ${\bs \chi}(K)$ (with initial condition $V(0)=K$) is obtained similarly.
Combining these two matrix equations yields the system that was stated in Lemma \ref{lemma: decomp chi(0) and chi(K)}.\hfill$\square$

\subsubsection*{Proof of Equation \eqref{eq: sys eqs nonsub}}
As noted before, the Laplace transform of the first term of the master equation simply follows from Lemma~\ref{lemma: eta and zeta Levy}. In order to obtain the Laplace transforms of the second and third terms, we first apply the change of variables $z\rightarrow w := u - y + z$ and then interchange the order of integration from $(\mathrm{d}\nu, \mathrm{d}w, \mathrm{d}y, \mathrm{d}u)$ to $(\mathrm{d}\nu, \mathrm{d}w, \mathrm{d}u, \mathrm{d}y)$.  This gives for the second term
\begin{align*}
    & {\bs 1}_{\{i\neq j\}}\frac{q_{ij}}{\omega_i}\,\psi_i(\omega_i)\int_0^\infty \mathbb{P}\left(\bar{Y}_i(T_{\omega_i})\in\mathrm{d}y\right) \int_y^\infty e^{-\gamma u} \int_{u-y}^\infty e^{-\psi_i(\omega_i) (w+y-u)} \int_{w}^\infty \mathbb{P}\left(B_{ij}\in\mathrm{d}\nu\right)\,e^{-\alpha(\nu-w)}\,{\mathrm d}w\,\mathrm{d}u,
\end{align*}
and for the third term
\begin{align*}
    &\:\:\:\:\:+ \sum_{k=1,k\neq i}^d \frac{q_{ik}}{\omega_i}\,\psi_i(\omega_i) \int_0^\infty  \mathbb{P}\left(\bar{Y}_i(T_{\omega_i})\in\mathrm{d}y\right) \int_y^\infty e^{-\gamma u} \int_{u-y}^\infty e^{-\psi_i(\omega_i) (w+y-u)}\\
    & \:\:\:\:\:\:\:\:\:\:\cdot\int_{0}^{w} \mathbb{P}\left(B_{ik}\in\mathrm{d}\nu\right)\,\eta_{kj}(w-\nu, \alpha,\beta)\,{\mathrm d}w\,\mathrm{d}u.
\end{align*}
Secondly, we apply the change of variables $u \to s := u-y$ and interchange order of integration 
\begin{equation*}
    (\mathrm{d}\nu, \mathrm{d}w, \mathrm{d}s, \mathrm{d}y) \to (\mathrm{d}\nu, \mathrm{d}s, \mathrm{d}w, \mathrm{d}y) \to (\mathrm{d}s, \mathrm{d}\nu, \mathrm{d}w, \mathrm{d}y) \to (\mathrm{d}s, \mathrm{d}w, \mathrm{d}\nu, \mathrm{d}y),
\end{equation*}
in that specific order, so that for the second term we obtain
\begin{align*}
    &{\bs 1}_{\{i\neq j\}}\frac{q_{ij}}{\omega_i}\,\psi_i(\omega_i)\int_0^\infty \mathbb{P}\left(\bar{Y}_i(T_{\omega_i})\in\mathrm{d}y\right) e^{-\gamma y} \int_0^\infty \mathbb{P}\left(B_{ij}\in\mathrm{d}\nu\right) e^{-\alpha\nu} \\
    &\:\:\:\:\:\cdot\int_0^\nu e^{-(\psi_i(\omega_i)-\alpha)w} \int_0^w e^{-(\gamma-\psi_i(\omega_i))s}\,\mathrm{d}s\,\mathrm{d}w,
\end{align*}
and for the third term
\begin{align*}
    &\sum_{k=1,k\neq i}^d \frac{q_{ik}}{\omega_i}\,\psi_i(\omega_i) \int_0^\infty \mathbb{P}\left(\bar{Y}_i(T_{\omega_i})\in\mathrm{d}y\right) e^{-\gamma y} \int_0^\infty \mathbb{P}\left(B_{ik}\in\mathrm{d}\nu\right) \\
    & \:\:\:\:\:\:\:\:\:\:\cdot\int_{\nu}^{\infty} e^{-\psi_i(\omega_i) w}\,\eta_{kj}(w-\nu, \alpha,\beta)\int_0^w e^{-(\gamma-\psi_i(\omega_i))s}\,\mathrm{d}s\,{\mathrm d}w.
\end{align*}
Lastly, we apply in the third term the change of variables $w \to \upsilon:= w-\nu$. The resulting quadruple integrals can be evaluated in the standard manner, so as to obtain \eqref{eq: sys eqs nonsub}. \hfill$\square$

\subsection{Notes on implementation}\label{app: notes on implementation}
This second appendix explains how we implemented our procedure to generate the figures presented in this paper and it discusses some of the challenges encountered. We emphasize, however, that this is not a comprehensive guide to implementing the procedure in full generality, as certain model instances involve additional subtleties. The specific cases studied in the numerical section allowed us to simplify the derivation of several objects and to obtain some inverse Laplace transforms in explicit form.

\medskip

To evaluate the workload moments, we do not require the full LST of $V(T_\beta)$ as given by the decomposition in Lemma \ref{lemma: decomp chi(x)}, but can instead derive them directly via a similar decomposition. Specifically, by taking the $n$-th derivative of ${\bs \chi}(x)$ with respect to $\alpha$ in \eqref{eq: chi(x)} and letting $\alpha\downarrow0$ (or, alternatively, by a direct argumentation), we find that the $n$-th moment obeys
\begin{align}
    \mathbb{E}_{i,x}&\left(V(T_\beta)^n\,{\bs 1}_{\{J(T_\beta = j)\}}\right) = \sum_{k=1}^d\delta_{-,ik}(x)\,\mathbb{E}_{k,0}\left(V(T_\beta)^n\,{\bs 1}_{\{J(T_\beta = j)\}}\right)\:+\notag \\ &\quad\sum_{k=1}^d\delta_{+,ik}\,(x)\,\mathbb{E}_{k,K}\left(V(T_\beta)^n\,{\bs 1}_{\{J(T_\beta = j)\}}\right)+ \mathbb{E}_{i,x}\left(V(T_\beta)^n\,{\bs 1}_{\{T_\beta<\min\{\sigma(x), \tau(K-x)\}, J(T_\beta) =j\}}\right), \label{eq: deco mom}
\end{align}
for $x\in[0,K], \beta>0$ and $i,j\in\defD $, with ${\bs \delta}_-(x)$ and ${\bs \delta}_+(x)$ as identified in Section \ref{sec: hitting probabilities}. 
The quantities
\[\mathbb{E}_{k,0}\left(V(T_\beta)^n{\bs 1}_{\{J(T_\beta = j)\}}\right)\:\:\mbox{and}\:\:\mathbb{E}_{k,K}\left(V(T_\beta)^n{\bs 1}_{\{J(T_\beta = j)\}}\right),\]
as appearing in the right-hand side of \eqref{eq: deco mom},
can be found by taking the $n$-th derivative with respect to $\alpha$ in the matrix equation of Lemma \ref{lemma: decomp chi(0) and chi(K)} and then letting $\alpha\downarrow0$. Noting that ${\bs P}$ does not depend on~$\alpha$, this leads to the same system of linear equations but with  adjusted  ${\bs b}^{\{0\}}$ and ${\bs b}^{\{K\}}$. Recalling the definition of ${\bs b}^{\{0\}}$ and ${\bs b}^{\{K\}}$, this adjustment involves (i) the $n$-th moment of  $V_i(T_\beta)$  (in the case of being killed before a transition of the background process), facilitated by Lemma \ref{lemma: LST W Levy inti x}, and (ii) the $n$-th derivative of ${\bs \delta}_\star(x)$ to $\alpha$ at $\alpha\downarrow 0$  in \eqref{eq: delta*}, bearing in mind that in the notation ${\bs \delta}_\star(x)$ that we introduced in \eqref{eq: def delta star} we have suppressed the dependence on $\alpha$.  Note that for (ii) we need ${\bs \eta}(u,\alpha,\beta)$ for $u\in\{K-x,K\}$, which is achieved by numerical inversion of ${\bs \eta}(\alpha,\beta, \gamma)$ (as identified in Section~\ref{sec: overshoot transform}).
The last term in the right-hand side of \eqref{eq: deco mom} can also be evaluated by taking the $n$-th derivative of ${\bs \delta}_\star(x)$ to $\alpha$ at $\alpha\downarrow 0$.

\medskip

Along the same lines, the probability of an empty system can be found. Using the short notation ${\mathscr E}_{j,0}:=\{V(T_\beta) = 0, J(T_\beta) = j\}$, we have the decomposition
\begin{align*}
    \mathbb{P}_{i,x}({\mathscr E}_{j,0}) =& \sum_{k=1}^d\delta_{-,ik}(x)\,\mathbb{P}_{i,0}({\mathscr E}_{j,0})+\sum_{k=1}^d\delta_{+,ik}(x)\,\mathbb{P}_{i,K}({\mathscr E}_{j,0})+ \mathbb{P}_{i,x}({\mathscr E}_{j,0}, T_\beta<\min\{\sigma(x), \tau(K-x)\}).
\end{align*}
It is noted that $\mathbb{P}_{i,0}({\mathscr E}_{j,0})$ and $\mathbb{P}_{i,K}({\mathscr E}_{j,0})$ can be computed by taking the limit $\alpha\to \infty$ in Lemma \ref{lemma: decomp chi(0) and chi(K)}; this again amounts to solving the linear system with adjusted ${\bs b}^{\{0\}}$ and ${\bs b}^{\{K\}}$. The right-most probability is equal to
\begin{equation}
    \mathbb{P} (Y(T_\beta) = -x, T_\beta<\min\{\sigma(x), \tau(K-x)\}, J(T_\beta) = j\,|\,J(0)=i); \label{eq: zero prb}
\end{equation}
{note that $Y(T_\beta)<-x$ cannot occur, since it would imply $\sigma(x)\leqslant T_\beta$, violating the condition $T_\beta<\min\{\sigma(x), \tau(K-x)\}$. Hence, \eqref{eq: zero prb} is zero for all $x\in(0,K]$: the considered event rules out that $Y$ is a subordinator, and when $Y$ is not a subordinator clearly ${\mathbb P}(Y(T_\beta)=-x)=0$.}

{For $x=0$, some care is required. When $i\in\defDmin $, that is, the initial state is a subordinator state, \eqref{eq: zero prb} is again zero, since in this case $\sigma(0)=0$ almost surely. When $i\in\defDplus $, however, this probability need not vanish. The reason is that $\sigma(0)=\infty$, while ${\mathbb P}(Y(T_\beta)=0)$, and hence also ${\mathbb P}(V(T_\beta)=0)$, may be positive. An example is a compound Poisson process without drift: here $\sigma(0)=\infty$, but $Y(T_\beta)$ can equal zero if no jumps and no background transitions occur before killing. Such a process was encountered in Instance~1 in Section \ref{sec: num}, where we obtained the probability in question by taking the limit $\alpha\to\infty$ of ${\bs\delta}_\star(x)$.}


{\small
\begin{thebibliography}{99}%

\bibitem{AW}
{\sc J. Abate} and {\sc W. Whitt} (1995).
Numerical inversion of Laplace transforms of probability distributions. {\it ORSA Journal on Computing} {\bf 7}, pp. 36-43.

\bibitem{AW2}
{\sc J. Abate} and {\sc W. Whitt} (1992). 
The Fourier-series method for inverting transforms of probability distributions. {\it Queueing Systems} {\bf 10}, pp. 5–87.

\bibitem{AAGP}
{\sc  L.N. Andersen, S. Asmussen, P.  Glynn,} and {\sc M.  Pihlsgaard} (2015). 
L\'evy processes with two-sided reflection. In: {\it L\'evy Matters V. Functionals of L\'evy Processes} (O. Barndorff-Nielsen, J. Bertoin, J. Jacod, and K. Kl\"uppelberg, eds.), pp. 67–182. Springer, New York.

\bibitem{AND}
{\sc L.N. Andersen} and {\sc M. Mandjes} (2009). 
Structural properties of reflected L\'evy processes. {\it Queueing Systems} {\bf  63}, pp.\ 301–322. 

\bibitem{ASG}
{\sc N.M. Asghari, P. den Iseger,} and {\sc M. Mandjes} (2014). Numerical techniques in Lévy fluctuation theory. {\it Methodology and Computing in Applied Probability} {\bf 16}, pp.\ 31-52.

\bibitem{ASM2}
{\sc S. Asmussen} (2003).
{\em Applied Probability and Queues}, 2nd ed.
Springer, New York.

\bibitem{AA}
{\sc S. Asmussen} and {\sc H. Albrecher} (2010).
{\em Ruin Probabilities}, 2nd ed.
World Scientific, Singapore.



\bibitem{BER}
{\sc J. Bertoin} (1996).
{\it L\'evy Processes}. Cambridge University Press, Cambridge. 


\bibitem{JWC}
{\sc J. Cohen} (1969). {\it The Single-Server Queue}, 1st ed. North Holland, Amsterdam. 


\bibitem{IV}
{\sc B. D'Auria, J. Ivanovs, O. Kella}, and {\sc M. Mandjes} (2010).
First passage of a Markov additive process and generalized Jordan chains.
{\it Journal of Applied Probability} {\bf 47}, pp.\ 1048-1057.

\bibitem{DM}
{\sc K. D\c{e}bicki} and {\sc M. Mandjes} (2015).
{\it Queues and L\'evy fluctuation theory.} Springer, New York. 

\bibitem{DIEK}
{\sc T. Dieker} and {\sc M. Mandjes} (2011). 
Extremes of Markov-additive processes with one-sided jumps, with queueing applications. {\it
Methodology and Computing in Applied Probability} {\bf  13}, pp. 221-267.

\bibitem{HOU}
{\sc B. Housley} (2024). 
Extended L\'evy’s theorem for a two-sided reflection.
{\it Electronic Communications in Probability} {\bf 29}, paper no.\ 15. 

\bibitem{dI}
{\sc P. den Iseger} (2006). 
Numerical transform inversion using Gaussian quadrature.
{\it Probability in the Engineering and  Informational Sciences} {\bf 20}, pp.\ 1-44.

\bibitem{IV2}
{\sc J. Ivanovs} (2010).
Markov-modulated Brownian motion with two reflecting barriers. 
{\it Journal of Applied Probability} {\bf 47}, pp.\ 1034-1047.

\bibitem{thesis_IV}
{\sc J. Ivanovs} (2011).
One-sided Markov additive processes and related exit problems.
PhD thesis, University of Amsterdam, {\footnotesize \url{
https://pure.uva.nl/ws/files/1408093/94456_0_Thesis.pdf}}

\bibitem{IBM}
{\sc J. Ivanovs, O.\ Boxma,} and {\sc M. Mandjes} (2010).
Singularities of the matrix exponent of a Markov additive process with one-sided jumps.
{\it Stochastic Processes and their Applications} {\bf 120}, pp.\ 1776-1794.

\bibitem{IP}
{\sc J. Ivanovs} and {\sc Z. Palmowski} (2012).
Occupation densities in solving exit problems for Markov additive processes and their reflections.
{\it Stochastic Processes and their Applications} {\bf 122}, pp.\ 3342-3360.

\bibitem{MKD}
{\sc L. van Kreveld, M. Mandjes}, and {\sc J.-P. Dorsman} (2022).
Extreme value analysis for a Markov additive process driven by a nonirreducible background chain. {\it Stochastic Systems} {\bf 12}, pp.\ 293-317.

\bibitem{KRU}
{\sc L. Kruk, J.  Lehoczky, K. Ramanan,} and {\sc S. Shreve} (2007). 
An explicit formula for the Skorokhod map on $[0,a]$. {\it Annals of  Probability} {\bf 35}, pp.\ 1740–1768. 

\bibitem{PAR}
{\sc A. Kuznetsov, A. Kyprianou,} and {\sc J. Pardo} (2012).
Meromorphic Lévy processes and their fluctuation identities.
{\it Annals of Applied Probability} {\bf 22}, pp.\ 1101-1135.

\bibitem{KUZ}
{\sc A. Kuznetsov, A. Kyprianou,} and {\sc V. Rivero} (2013).
The theory of scale functions for spectrally
negative L\'evy processes.
In: {\it L\'evy Matters II. Recent Progress in Theory and Applications: Fractional Lévy Fields, and Scale Functions} (S. Cohen, ed.), pp. 97-186. Springer, New York.

\bibitem{KYP}
{\sc A.  Kyprianou} (2006). 
{\it Introductory Lectures on Fluctuations of L\'evy Processes with Applications}. Springer, New York.

\bibitem{KR}
{\sc A. Kyprianou} and {\sc V. Rivero} (2026).
The strong law of large numbers and a functional central limit theorem for general Markov additive processes. {\it Queueing Systems}, to appear.
{\footnotesize \url{
https://arxiv.org/abs/2505.10956}}


\bibitem{LM}
{\sc A. Lewis} and {\sc E. Mordecki} (2008). Wiener-Hopf factorization for L\'evy processes having positive jumps with rational transforms. {\it Journal of Applied Probability} {\bf 45}, pp.\ 118–134.

\bibitem{MB}
{\sc M. Mandjes} and {\sc O. Boxma} (2023).
{\it The Cramér-Lundberg Model and Its Variants}. Springer, New York.

\bibitem{PI}
{\sc M. Pistorius} (2003).
On doubly reflected completely asymmetric L\'evy processes. {\it Stochastic Processes and their Applications} {\bf 107}, pp.\ 131-143.


\bibitem{R1}
{\sc P. Roes} (1970).
The finite dam. {\it Journal of Applied Probability} {\bf 7}, pp.\ 316-326. 

\bibitem{R2}
{\sc P. Roes.} (1970).
The finite dam II.
{\it Journal of Applied Probability} {\bf 7}, pp.\ 599-616.

\bibitem{SUR}
{\sc B. Surya} (2008). 
Evaluating scale functions of spectrally negative L\'evy processes.  {\it Journal of Applied Probability} {\bf 45}, pp.\ 135-149.

\bibitem{TAK}
{\sc L. Tak\'acs} (1967). 
{\it  Combinatorial Methods in the Theory of Stochastic Processes.} Wiley, New York.

\end{thebibliography}}
\end{document}